\documentclass{article}
\usepackage{amsmath}
\usepackage{amssymb}
\usepackage{titlesec}
\usepackage{graphicx}
\usepackage{fancyhdr}
\usepackage{latexsym}
\usepackage{mathrsfs}
\usepackage{booktabs}
\usepackage{mathrsfs}
\usepackage{mathtools}
\usepackage{amsthm}
\usepackage{wasysym}
\usepackage{cite}
\usepackage{color}
\usepackage{ulem}
\usepackage{todonotes}
\usepackage{amsfonts}
\usepackage{graphics}
\usepackage{multimedia}
\usepackage{amscd}
\numberwithin{equation}{section}

\newtheorem{theorem}{Theorem}[section]
\newtheorem{lemma}{Lemma}[section]
\newtheorem{remark}{Remark}[section]
\newtheorem{proposition}{Proposition}[section]

\newcommand{\dv}{\text{div}}

\title{Global strong solution of the 3D inhomogeneous liquid crystal flows with density-dependent viscosity and large velocity}
\author{Jiaxu L{\small I}$^{b}$, Yu M{\small EI}$^c$, Rong Z{\small HANG}$^{a,d}$ \thanks{Email addresses:  Jiaxvlee@gmail.com (J. X. Li), yu.mei@nwpu.edu.cn (Y. Mei), rzhang0921@gmail.com (R. Zhang). }  \\ 
{\normalsize a. School of Mathematics and Computer Sciences,}\\
{\normalsize Nanchang University, Nanchang 330031, P. R. China;}\\
{\normalsize b.  The Institute of Mathematical Sciences,}\\
{\normalsize  The Chinese University of Hong Kong,
Shatin, N.T.;}\\
{\normalsize c. School of Mathematics and Statistics,}\\
{\normalsize Northwestern Polytechnical University, Xi'an Shaanxi, 710129, P. R. China}\\
{\normalsize d. Institute of Mathematics and Interdisciplinary Sciences,}\\
{\normalsize Nanchang University, Nanchang 330031, P. R. China.}
}
\date{}

\begin{document}
\maketitle
\begin{abstract}
This paper concerns the initial boundary value problem of three-dimensional inhomogeneous incompressible liquid crystal flows with density-dependent viscosity. When the viscosity coefficient $\mu(\rho)$ is a power function of the density with the power larger than $1$, that is $\mu(\rho)=\mu\rho^\alpha$ with $\alpha>1$, it is proved that the system exists a unique global strong solution as long as the initial density is sufficiently large and $L^3$-norm of the derivative of the initial director is sufficiently small.  This is the first result concerning the global strong solution for three-dimensional inhomogeneous liquid crystal flows without smallness of velocity.

\end{abstract}

\section{Introduction}

Nematic liquid crystals are intermediate states between solid crystals and isotropic liquids, which consist of rod-like molecules with a preferred local average direction. The static configurations of nematic liquid crystals are usually described by a unit vector $d\in\mathbb{S}^2$ being subject to the variation of the Oseen-Frank energy proposed in \cite{oseen1933theory,frank1958liquid}. Based on this static vector model and balance laws from fluid mechanics, Ericksen\cite{ericksen1961conservation} and Leslie\cite{leslie1968some} introduced the widely accepted hydrodynamic model for nematic liquid crystal flows. Precisely, the inhomogeneous incompressible general Ericksen-Leslie system, as a model for the motion of a multi-phase nematic liquid crystal flows with different densities, reads as follows
\begin{equation}\label{GILC}
			\left\{\begin{array}{lr}	
   \rho_t+\mathrm{div}(\rho u)=0,\\
   (\rho u)_t+\mathrm{div}(\rho u\otimes u)+\nabla P=\mathrm{div}(\sigma^E+\sigma^L)\\
d\times\left(\lambda_1 N+\lambda_2Ad-h\right)=0,\\
\mathrm{div} u=0,\quad |d|=1,
\end{array}\right.
\end{equation}
where $\rho$, $u$, $P$ and $d$ represent, respectively, the density, the velocity, the pressure, and the director of liquid crystal flow. $\sigma^E$ denotes the Ericksen stress tensor  given by
\begin{equation}\label{E-T}
\sigma^E=-\nabla d^\top\frac{\partial W(d,\nabla d)}{\partial  (\nabla d)}
\end{equation}
with the Oseen-Frank density $W(d,\nabla d)$ of the form
\begin{equation*}
W(d,\nabla d)=k_1(\text{div } d)^2+k_2 (d\cdot\text{curl }d)^2+k_3
|d\times \text{curl } d|^2+(k_2+k_4)[\text {tr} (\nabla d)^2-(\text {div
}d)^2],
\end{equation*}
where $k_1,k_2,k_3,k_4$ are the Frank elastic constants. $\sigma^L$ is the Leslie stress tensor of the form
\begin{equation}\label{L-T}
\sigma^L=\mu_1(d\otimes d:\mathcal{D}(u))d\otimes d+\mu_2 N\otimes d+\mu_3 d\otimes N+\mu_4 \mathcal{D}(u)+\mu_5 (\mathcal{D}(u)d)\otimes d+\mu_6 d\otimes (\mathcal{D}(u)d),
\end{equation}
where $\mu_i$, $i=1,2,\cdots, 6$ are the Leslie coefficients. The co-rotational time derivative $N$ of $d$  is defined by
\begin{equation}\label{co-der}
N=d_t+u\cdot\nabla d-\mathcal{A}(u) d.
\end{equation}
Here $\mathcal{D}(u)$ and $\mathcal{A}(u)$ are the deformation and rotation tensor respectively, that is
\begin{equation}
    \mathcal{D}(u)=\frac12 \left(\nabla u+(\nabla u)^\top\right),\qquad \mathcal{A}(u)=\frac12 \left(\nabla u-(\nabla u)^\top\right).
\end{equation}
The molecular field $h$ in \eqref{GILC} is given by
\begin{equation}\label{M-F}
h=\mathrm{div}\left(\frac{\partial W(d,\nabla d)}{\partial (\nabla d)}\right)-\frac{\partial W(d,\nabla d)}{\partial d}.
\end{equation}
The general Ericksen-Leslie system \eqref{GILC} is very complicated so mathematical studies of the full system are very challenging. Considering the one-constant approximation of the Oseen-Frank energy density, that is, $k_1=k_2=k_3,k_4=0$ and ignoring the effect of Leslie tensor due to director, one can reduce \eqref{GILC} to the following simplified system, which persist the essential strong coupling and nonlinearity,
\begin{equation}\label{ins}
	\begin{cases}
	\rho_t+\mathrm{div}(\rho u)=0,\\
		(\rho u)_t+\mathrm{div}(\rho u\otimes u)+\nabla P-\mathrm{div}(2\mu\mathcal{D}(u))=-\nu\mathrm{div}(\nabla d\odot\nabla d), \\
        d_t+u\cdot\nabla d=\lambda(\Delta d+|\nabla d|^2d),\\
        \dv u=0,|d|=1,
	\end{cases} 
\end{equation}
where $\mu$ is the fluid viscosity coefficient, $\nu>0$ is the Frank coefficient and  $\lambda>0$ stands for the microscopic elastic relaxation time. $\nabla d\odot\nabla d$ is the elastic tensor with $ij$-component $\partial_id\cdot\partial_j d$. Even for the simplified system \eqref{ins}, it is a strongly coupled system between the inhomogeneous incompressible
Navier–Stokes equations and the transported harmonic heat flow. The global existence, uniqueness, and regularities of this simplified system are fascinating and difficult issues in the mathematical theory of liquid crystal flows.     

In the last fifteen years, there has been some great progress for homogeneous liquid crystal flows. For the two-dimensional problem, the global existence of weak solutions with finitely many singular times is proved in \cite{lin2010liquid,hong2011global,hong2012global,huang2014regularity,wang2014global} for both simplified and general Ericksen-Leslie systems. The uniqueness of those weak solutions is established in \cite{li2016uniqueness,wang2016uniqueness} and reference therein. However, the regularity of those two-dimensional weak solutions can not be expected in general. Lei-Li-Zhang\cite{lei2014remarks} established a rigidity theorem for harmonic maps under a geometric angle condition and obtained global wellposedness of smooth solutions to a simplified system for a class of large initial data. Recently, Lai-Lin-Wang-Wei-Zhou\cite{lai2022finite} developed a new inner-outer gluing method to construct solutions to a simplified system
that blow up exactly at any given finite points in a bounded domain as $t$ goes to a finite time $T$. For the three-dimensional problem, global existence of weak solutions is a long-standing open problem. Lin-Wang\cite{lin2016global} solved the problem for the simplified system in the case that if the initial director satisfies the semi-sphere condition $d_0\in \mathbb{S}^2_+$. It is still open for both the simplified system without this geometric condition and the general Ericksen-Leslie system. If the initial data have more regularities, we may not obtain global well-posed theory in general, refer to \cite{huang2016finite} for finite blow-up results. The local well-posedness and blow-up criteria of strong solutions are established in \cite{huang2012blow,hong2014blow,hieber2016dynamics,wang2013well,wang2014global}. For rough initial data,  Hineman-Wang \cite{hineman2013well} established the local well-posedness of solutions to the simplified system with both initial velocity and gradient of director in uniformly local  $L^3$-integrable spaces. This result is generalized by Hong and the second author \cite{hong2019well} to the Ericksen-Leslie system with Oseen-Frank energy through a different idea, which is used to overcome the invalid of direct $L^3$-type energy estimates because of the nonlinear structure of the Oseen-Frank energy.  

When considering inhomogeneous incompressible liquid crystal flows, global existence of weak solutions is unknown for both two and three-dimensional problems even for the simplified system \eqref{ins}, because the local energy estimate method developed in \cite{lin2010liquid,hong2011global} is invalid. There are only some well-posedness of strong solutions results for \eqref{ins}. Wen and Ding\cite{wen2011solutions} firstly established  the local existence and uniqueness of strong solutions to the 2D and 3D Cauchy problem with vacuum and extended it to a global one in 2D case with positive initial density and small initial basic energy. Li\cite{li2014global} obtained local well-posedness of 2D Cauchy problem, and the global well-posedness in the case that the initial density is positive and the initial direction field satisfies a geometric angle condition as in \cite{lei2014remarks}. Later, Gong-Li-Xu \cite{gong2016local} obtained local well-posedness of the 3D Cauchy problem by a biharmonic regularization approach.   When the viscosity $\mu$ depends on the density, Gao-Tao-Yao \cite{gao2016strong} proved the local well-posedness and blow-up criteria for the initial boundary value problem. For the global well-posedness, Li-Liu-Zhong\cite{li2017global} obtained  the result for
2D Cauchy problem provided the initial density and the 
gradient of director decay, which was not too slow at the far field, and the basic energy 
sufficiently small. This smallness assumption on basic energy was instead by the geometric condition on the initial director in \cite{liu2016global}. Hu-Liu\cite{hu2018global} proved the
result for the 3D Cauchy problem with small density, velocity, and director in Besov spaces. Very recently, Ye-Zhu\cite{ye2024existence} obtained global strong solution to 3D Cauchy problem with density-dependent viscosity provided the initial norm $\|u_0\|_{\dot{H}^s}+\|\nabla d_0\|_{\dot{H}^s}(\frac{1}{2}<s\leq 1)$ is suitably small. 

The global well-posedness results for strong solutions to \eqref{ins} in the three-dimension mentioned above all require some smallness assumption on both velocity and director. Very recently, Huang, the first and last authors \cite{huang2024global} surprisingly obtained a global large solution to three-dimensional inhomogeneous incompressible Navier-Stokes with density-dependent viscosity $\mu(\rho)=\rho^\alpha,$ as long as $\alpha>1$ and the initial data are sufficiently large. In fact, in this case, the Reynolds number $Re=\rho u L/\mu(\rho)$ will be small provided $\alpha>1$ and the density $\rho$ is large enough, no matter how fast the flow speed is.
At low Reynolds numbers, the flows behave like a laminar flow which is believed to be stable as many experiments revealed. It motivates that this Navier-Stokes flow may be globally stable as long as the density is large at any
time. Inspired by the idea in \cite{huang2024global}, it is interesting to investigate whether we can remove the smallness assumption on the velocity in the global well-posedness results for inhomogeneous liquid crystal flow \eqref{ins} with density-dependent viscosity $\mu(\rho)=\rho^\alpha,(\alpha>1)$. It should be emphasized that it seems impossible to further remove the smallness of the director by the approach in \cite{huang2024global}, since the director is only transported with the velocity and the decay of velocity obtained in \cite{huang2024global} for Navier-Stokes part can not be used to improve the regularity of director. About other important results on inhomogeneous Naiver-Stokes, refer to\cite{desjardins1997regularity,abidi2015global,abidi2015global1,cho2004unique,lu2019local,zhang2015global,he2021global,craig2013global,huang2014global,huang2015global}. On the other hand, the system \eqref{ins} admits the following scaling invariance in the sense that, if $(\rho,u,d)$ is a solution to \eqref{ins}, then for any $\tau>0$, 
\begin{equation}\label{scaling}
    \rho_\tau(x,t)=\rho(\tau x,\tau^2 t),~~ u_\tau(x,t)=\tau u(\tau x,\tau^2 t), ~~d_\tau(x,t)=d(\tau x,\tau^2t)
\end{equation}
is still a solution to \eqref{ins}. There are many works on the well-posedness of homogeneous or inhomogeneous Navier-Stokes in critical functional spaces with norms invariant under the scaling \eqref{scaling} for $(\rho,u)$, see \cite{fujita1964navier,zhang2020global,hao2024global} and the reference therein. As to liquid crystal flows, it should be remarked that \cite{ye2024existence} require smallness in sub-critical spaces and \cite{hu2018global} require smallness in critical space but for all the density, velocity and director.
It is interesting whether we can only ask for the smallness of the norm of director in some critical spaces to get the global well-posedness of strong solutions to \eqref{ins}.

In this paper, we concern the inhomogeneous incompressible simplified Ericksen-Leslie system \eqref{ins} for liquid crystal flows in a three-dimensional bounded domain $\Omega\subset\mathbb{R}^3$ with density-dependent viscosity $\mu(\rho)$,
% \begin{equation}
% 	\begin{cases}
% 		\rho_t+\mathrm{div}(\rho u)=0,&~~\mathrm{in}~\Omega\times [0,T],\\
% 		(\rho u)_t+\mathrm{div}(\rho u\otimes u)+\nabla P-\mathrm{div}(2\mu(\rho)\mathcal{D}(u))=-\nu\mathrm{div}(\nabla d\odot\nabla d),&~~\mathrm{in}~\Omega\times [0,T], \\
%         d_t+u\cdot\nabla d=\lambda(\Delta d+|\nabla d|^2d),&~~\mathrm{in}~\Omega\times [0,T],\\
%         \dv u=0,|d|=1&~~\mathrm{in}~\Omega\times [0,T]\\
%     \rho(0,x)=\rho_0(x),~~u(0,x)=u_0(x),&~~\mathrm{in}~\Omega,
% 	\end{cases} 
% \end{equation}
% where $t\geq0, x=(x_1,x_2,x_3)\in\Omega$ are time and space variables, respectively. $\rho=\rho(x,t)$, $u=(u_1(x,t),u_2(x,t),u_3(x,t))$, $P$ and $d=(d_1(x,t),d_2(x,t),d_3(x,t))$ represent, respectively, the density, the velocity, the pressure and the director of liquid crystal flow.
% \begin{equation}
% 	\mathcal{D}(u)=\frac12 \left(\nabla u+(\nabla u)^\top\right),
% \end{equation}
% is the deformation tensor. $\nabla d\odot\nabla d$ is the elastic tensor with $ij$-component $\partial_id\cdot\partial_j d$. Here $\mu(\rho)$ stands for the viscosity and is a function of $\rho$, 
which is assumed to satisfy
\begin{equation}\label{vis-d}	\mu(\rho)=\mu\rho^\alpha,\ \mu>0,\ \alpha>0. 
\end{equation}
We will investigate the global existence and uniqueness of strong solutions to the initial boundary value problem to the system \eqref{ins}, \eqref{vis-d} with the following Dirichlet boundary condition for the velocity and Neumann boundary condition for the director
\begin{equation}\label{bc}
    u=0,~~\frac{\partial d}{\partial n}=0~~\mathrm{on}~\partial\Omega\times [0,T]
\end{equation}
and initial data
\begin{equation}\label{ini-data}
\rho(0,x)=\rho_0(x),~~u(0,x)=u_0(x)~~d(0,x)=d_0(x), ~~~~x\in\Omega.
\end{equation}

Before stating the main results, we explain the notation and conventions used throughout this paper. 
Denote
\begin{equation*}
	\int f \mathrm{dx}=\int_{\Omega} f\mathrm{dx}.
\end{equation*}
For a positive integer $k$ and $p\geq1$, we denote the standard Lesbegue and Sobolev spaces as follows:
\begin{equation*}
\begin{gathered}
\|f\|_{L^p}=\|f\|_{L^p(\Omega)},\ \|f\|_{W^{k,p}}=\|f\|_{W^{k,p}(\Omega)},\ \|f\|_{H^k}=\|f\|_{W^{k,2}(\Omega)},\\
C^\infty_{0,\sigma}=\{f\in C^\infty_0(\Omega):\mathrm{div}f=0\}, \ H_0^1=\overline{C^\infty_0}, \\
H_{0,\sigma}^1=\overline{C^\infty_{0,\sigma}},\ \mathrm{closure}\ \mathrm{in}\ \mathrm{the}\ \mathrm{norm}\ \mathrm{of}\  H^1.    
\end{gathered}
\end{equation*}

Our main result is stated as follows.
\begin{theorem}\label{global} 
    Let $\Omega$ be a bounded smooth domain in $\mathbb{R}^3$. Assume that
    \begin{equation}\label{vis-c}
        \alpha>1.
    \end{equation}
    Given constants
    \begin{equation}
        \bar\rho> 1,\quad C_0\geq 1.
    \end{equation}
    Suppose that the initial data $(\rho_0,u_0)$ satisfies 
    \begin{equation}\label{ia}
		\bar \rho \le \rho_0\le C_0 \bar \rho,\quad \rho_0 \in %L^{\gamma} \cap D^{1,2} \cap 
		W^{1,q},\ 3<q<6,\quad  u_0 \in H_{0,\sigma}^1 \cap H^2,\quad d_0\in H^3.
    \end{equation}
    Then there exists positive constants $\Lambda_0, \varepsilon_0$ depending only
	on $C_0,\mu, \nu, \lambda, \alpha,$ $\|\nabla\rho_0\|_{L^q}, \| u_0\|_{H^1},$ $ \|\nabla d_0\|_{H^1},$ and $\Omega$ such that if
	\begin{equation}\label{ini}
		\bar\rho\ge \Lambda_0, \quad \|\nabla d_0\|_{L^3}^2 \le \varepsilon_0,
	\end{equation}
    then the inhomogeneous incompressible simplified Ericksen-Leslie system  (\ref{ins})-(\ref{bc}) admits a unique global strong solution $(\rho,u)$ in $\Omega\times(0,\infty)$ satisfying 
    \begin{equation}
    \left\{ \begin{array}{l}
    \rho\in C([0,\infty);W^{1,q}),\
    \nabla u, P\in C([0,\infty);H^1)\cap L^2((0,\infty);W^{1,q}),\\
    \rho_t \in C([0,\infty);L^q),\
    \sqrt{\rho}u_t\in L^\infty((0,\infty);L^2),\ 
    u_t\in L^2((0,\infty);H^1_0),\\
    d\in C([0,T;H^3])\cap L^2(0,T;H^4), d_t\in C([0,T;H^1])\cap L^2(0,T;H^2).
    \end{array} \right.
\end{equation}
\end{theorem}

\begin{remark}
    Conditions \eqref{vis-c} and \eqref{ini} imply that the initial Reynolds number is suitably small, therefore the conclusion of Theorem 1.1 is equivalent to proving a fact that the flow is globally stable when the initial Reynolds number and the forcing term in the momentum equation involving the director $d$ are small enough. 
\end{remark}
\begin{remark}
    The smallness of critical norm $\|\nabla d_0\|_{L^3}$ plays an important role in Theorem \ref{global}. It is interesting to consider other critical space, for example, $\dot{H}^\frac{1}{2}$, instead of $L^3$ space of $\nabla d$. We will study this problem in the future. 
\end{remark}
\begin{remark}
    Theorem \ref{global} can also be applied to boundary condition \eqref{ins}$_4$ when replaced by Navier-slip boundary conditions and periodic ones. However, the Cauchy problem presents essential difficulties that will be addressed in future work.
\end{remark}
\begin{remark}
    Theorem \ref{global} can be generalized to the inhomogeneous general Ericksen-Leslie system \ref{GILC} in a three-dimensional periodic domain. It can also be applied to the two-dimensional problem with either $\|\nabla d_0\|_{L^2}$ small or the geometric condition $d_{03}\geq c_0>0$. 
\end{remark}

We now provide our analysis and commentary on the key aspects of this paper. The key observation is inspired by the definition of the Reynolds number, where we can remove the smallness of initial velocity by making the initial density sufficiently large. 
The main idea is to use time-weighted energy estimates to the compressible Navier–Stokes equations established by Hoff \cite{hoff1995global}, which is successfully used to the inhomogeneous incompressible Navier-Stokes equations by Huang-Wang \cite{craig2013global,huang2014global,huang2015global} and many other works, see \cite{paicu2013global} and references therein. In order to apply this idea to the inhomogeneous liquid crystal flows, it is crucial to deal with the extra term $\mathrm{div}(\nabla d\odot\nabla d)$, which is supercritical and coupled with the estimates of director. Motivated by Hong and the second authors' work \cite{hong2019well} and Hineman-Wang \cite{hineman2013well} on homogeneous liquid crystal flows, it is essential to utilize the smallness of $\|\nabla d\|_{L^\infty_tL^3_x}$ to obtain the dissipation of $d$ so that to control the term involving $\mathrm{div}(\nabla d\odot\nabla d)$ in the process of time-weighted energy estimate of $u$. Moreover, this smallness assumption of $\|\nabla d\|_{L^\infty_tL^3_x}$ may not be verified directly, as in \cite{hineman2013well} for the Cauchy problem, by deriving $L^3$-type energy estimate of $\nabla d$ in case that the boundary integrals from integration by parts can not be controlled. Instead, as in \cite{hong2019well}, we derive $L^2$-type energy estimate of $\nabla d$ and $\nabla^2d$ and use the interpolation inequality to close the smallness assumption of $\|\nabla d\|_{L^\infty_tL^3_x}$.

Let's briefly sketch the proof. First we assume that $\mathcal{E}_d(T)$, $\mathcal{E}_\rho(T)$, and $\mathcal{E}_u(T)$ defined in \eqref{As1}, \eqref{As2}, and \eqref{As3} are less than $2\delta$, 3$\mathcal{E}_\rho(0)$,  and 
3$\mathcal{E}_u(0)$ respectively, then we prove that in fact $\mathcal{E}_d(T)$ is less than $\delta$, $\mathcal{E}_\rho(T)$ is less than 2$\mathcal{E}_\rho(0)$, and $\mathcal{E}_u(T)$ is less than 2$\mathcal{E}_u(0)$ under the assumption that the initial density is large enough as well as $\|\nabla d_0\|_{L^3}$ is small enough. On the other hand, the control of $\|\nabla u\|_{L^1_tL^\infty_x}$ leads to uniform estimates for other higher-order quantities, which guarantees the extension of local strong solutions.
One of the main ingredients is a time-independent estimate which is essential due to exponential time decay estimates for $u$ and $d$ in a bounded domain.

\section{Preliminaries}
First, the following local existence theory, where the initial density is strictly away from vacuum, can be shown by similar arguments as in Gao-Tao-Yao \cite{frank1958liquid}:
\begin{lemma}\label{local}
    Assume that the initial data $(\rho_0, u_0, d_0)$ satisfies the regularity condition \eqref{ia}.
  %   and the compatibility condition
  %   \begin{equation}\label{cc}
		% -\mathrm{div}\left(\mu \rho_0^\alpha \bigl(\nabla u_0+\nabla^\bot u_0\bigr)\right)+\nabla P_0=\rho_0^\frac{1}{2}g,
  %   \end{equation}
  %   for some $(P_0,g)\in H^1\times L^2$. 
    Then there exists a small time $T$ and a unique strong solution $(\rho, u, d, P)$ to the initial boundary value problem \eqref{ins}-\eqref{ini-data} such that
\begin{equation}\label{l-r}
    \left\{ \begin{array}{l}
    \rho\in C([0,T];W^{1,q}),\
    \nabla u, P\in C([0,T];H^1)\cap L^2([0,T];W^{1,q}),\\
    \rho_t \in C([0,T];L^q),\
    \sqrt{\rho}u_t\in L^\infty([0,T];L^2),\ 
    u_t\in L^2([0,T];H^1_0),\\
    d\in C([0,T;H^3])\cap L^2(0,T;H^4), d_t\in C([0,T;H^1])\cap L^2(0,T;H^2).
    \end{array} \right.
\end{equation}
% Furthermore, if $T^\ast$is the maximal existence time of the local strong solution $(\rho, u)$, then either
% $T^\ast=\infty$ or
% \begin{equation}\label{blow-up}
%     \sup_{0\leq t\leq T^\ast}\left(\|\nabla\rho\|_{L^q}+\|\nabla u\|_{L^2}\right)=\infty.
% \end{equation}
\end{lemma}
In this paper, we will employ Bovosgii's theory which can be found in \cite{seregin2014lecture}.
\begin{lemma}\label{bovosgii}
    Let $\Omega$ be a bounded domain with Lipschitz boundary, $1 < p < \infty$. Given $b \in L^p(\Omega)$ with $\int_\Omega b dx=0$, there exists $v \in W_0^{1,p}(\Omega)$ with the following properties:
    \begin{equation*}
        \mathrm{div} v=b,
    \end{equation*}
    in $\Omega$, and
    \begin{equation}
        \|\nabla v\|_{L^p(\Omega)}\leq C(p)\|b\|_{L^p(\Omega)}.
    \end{equation}
\end{lemma}
Next, the following Poincar\'e type inequality can be found in \cite{galdi2011introduction}.
\begin{lemma}
    Let $\Omega$ be bounded and $C^1$-smooth, and let $u$ be a vector function with components in $W^{1,p}(\Omega)$, $1 \leq p<\infty$, and $u\cdot n =0$ at $\partial\Omega$. Then
    \begin{equation}
        \|u\|_{L^p(\Omega)}\leq C\|\nabla u\|_{L^p(\Omega)},
    \end{equation}
    where the constant $C$ depends only on $n$, $q$ and $\Omega$.
\end{lemma}
Also, the well-known Gagliardo-Nirenberg inequality \cite{ladyzhenskaia1968linear} will be frequently used in this paper.
\begin{lemma}\label{G-N}
Assume that $\Omega$ is a bounded Lipschitz domain in $\mathbb{R}^3$. Let 1 $\leq q \le  +\infty$ be a positive extended real quantity. Let j and m  be non-negative integers such that $ j < m$. Furthermore, let $1 \leq r \leq \infty$ be a positive extended real quantity,  $p \geq 1$  be real and  $\theta \in [0,1]$  such that the relations
\begin{equation}
    \dfrac{1}{p} = \dfrac{j}{n} + \theta \left( \dfrac{1}{r} - \dfrac{m}{n} \right) + \dfrac{1-\theta}{q}, \qquad \dfrac jm \leq \theta \leq 1
\end{equation}
hold. Then, 
\begin{equation}\label{GN-p}
    \|\nabla^j u\|_{L^p(\Omega)} \leq C\|\nabla^m u\|_{L^r(\Omega)}^\theta\|u\|_{L^q(\Omega)}^{1-\theta} + C_1\|u\|_{L^q(\Omega)},
\end{equation}
where $u \in L^q(\Omega)$  such that  $\nabla^m u \in L^r(\Omega)$. Moreover, if $q>1$ and $r>3$,
\begin{equation}\label{GN-C}
\|u\|_{C(\bar{\Omega})}\leq C\|u\|^{q(r-3)/(3r+q(r-3))}_{L^q}\|\nabla u\|^{3r/(3r+q(r-3))}_{L^r}+C_2\|u\|_{L^q}.
\end{equation}
where $u \in L^q(\Omega)$  such that  $\nabla u \in L^r(\Omega)$.
In any case, the constant $C > 0$  depends on the parameters $j,\,m,\,n,\,q,\,r,\,\theta$, on the domain $\Omega$, but not on $u$.

In addition, if $u\cdot n|_{\partial\Omega} = 0$ or $\int_\Omega u dx=0$, we can choose $C_1 = C_2 = 0$.
\end{lemma}

We also frequently use the following $L^p$-estimates for two elliptic systems, whose proof is a direct consequence of a well-known elliptic theory due to Agmon, Douglis, and Nirenberg \cite{agmon1964estimates}. One is the Neumann problem
\begin{lemma}\label{Nuemann_elliptic}
   Let $\Omega$ be a bounded domain with $C^{k+2}$ boundary. Given $f\in H^k(\Omega)$, the solution to the following Neumann boundary problem of the elliptic equation
   \begin{equation}
       \begin{cases}
           \Delta v=f,&\text{ in }\Omega,\\
           \frac{\partial v}{\partial n}=0,&\text{ on }\partial\Omega
       \end{cases}
   \end{equation}
   satisfies the following elliptic estimates
   \begin{equation}
       \|\nabla v\|_{H^{k+1}}\leq C_3\|f\|_{H^k}.
   \end{equation}
\end{lemma}
The other is the slip boundary problem
\begin{equation}\label{slip}
       \begin{cases}
        \Delta u=f,&\text{ in }\Omega,\\
		u\cdot n = 0,\ \mathrm{curl}u\times n=0,&\text{ on }\partial\Omega.
       \end{cases}
\end{equation}
\begin{lemma}\label{slip_estimate}
Let $u$ be a smooth solution of the elliptic system \eqref{slip}, then for $p\in(1,\infty)$, $k\geq 0$, there exists a positive constant $C$ depending only on $\lambda$, $\mu$, $p$ and $k$ such that
\begin{equation}
\|\nabla^{k+2}u\|_{L^p(B_R)}\leq C\|f\|_{W^{k,p}(B_R)}.
\end{equation}
\end{lemma}

\section{A priori estimates}

For any fixed time $T>0$, $(\rho,u,P, d)$ is the unique local strong solution to \eqref{ins}-\eqref{bc} on $\Omega\times (0,T]$ with initial data $(\rho_0,u_0, d_0)$ satisfying \eqref{ia}, which is guaranteed by Lemma 
\ref{local}.

Define
\begin{gather}\label{As1}
    \mathcal{E}_d(T)\triangleq\sup_{t\in[0,T] }\|\nabla 
    d \|_{L^3}^2,\\
    \label{As2}\mathcal{E}_u(T) \triangleq \sup_{t\in[0,T] } \bar\rho^{\alpha}\|\nabla u \|_{L^2}^2 
    +\int_0^{T}  \|\sqrt{\rho} u_t \|_{L^2}^2\,dt,\\
    \label{As3}\mathcal{E}_\rho(T) \triangleq \sup_{t\in[0,T] }\|\nabla \rho \|_{L^q}.
\end{gather}

We have the following key proposition.

\begin{proposition}\label{pr}
Under the conditions of Theorem \ref{global}, there exist  positive constants  $\Lambda_0, \varepsilon_0$ depending on $\Omega,C_0,\mu,\nu, \lambda,  \alpha$, and $\|\nabla\rho_0\|_{L^q}, \| u_0\|_{H^1}, \|\nabla d_0\|_{H^1}$ such that if $(\rho,u,P,d)$ is a smooth solution to the problem (\ref{ins})--(\ref{bc}) on $\Omega\times (0,T]$ satisfying
    \begin{equation}\label{a1}
        \mathcal{E}_d(T)\leq 2\delta,
		~\mathcal{E}_u(T) \le 3\mathcal{E}_u(0),
        ~\mathcal{E}_\rho(T) \le 3\mathcal{E}_\rho(0),
    \end{equation}
for some $\delta<1$ sufficiently small (defined in \eqref{delta}), then the following estimates hold:
    \begin{equation}
        \mathcal{E}_d(T)\le \delta,
		~\mathcal{E}_u(T) \le 2\mathcal{E}_u(0),
        ~\mathcal{E}_\rho(T) \le 2\mathcal{E}_\rho(0),
    \end{equation}
	provided 
    \begin{equation}
		\bar\rho\ge \Lambda_0,\quad \|\nabla d_0\|_{L^3}\leq  \varepsilon_0.
    \end{equation}     
\end{proposition}
First, as the density satisfies the transport equation \eqref{ins}$_1$ and making use of \eqref{ins}$_4$, one has the following lemma
\begin{lemma}
It holds that 
	\begin{equation}\label{2.3}
		\bar \rho \le \rho \le C_0 \bar \rho,\quad (x,t)\in \Omega \times [0,T].
	\end{equation}
\end{lemma}
Next, the basic energy inequality of the system \eqref{ins} reads
\begin{lemma}
    It holds that
    \begin{align}\label{basic-est-0}
		&\sup_{0\le t\le T} \left(\bar \rho \| u\|_{L^2}^2+\|\nabla d\|_{L^2}^2\right)   +  \int_{0}^{T} \bar\rho^{\alpha}\|\nabla u\|_{L^2}^2+\|\Delta d+|\nabla d|^2d\|_{L^2}^2 dt \nonumber\\
        &\le C(\bar\rho\|u_0\|_{L^2}^2+\|\nabla d_0\|_{L^2}^2) \le C \bar \rho,
	\end{align}
 where $C$ depends on $\|u_0\|_{L^2},\|\nabla d_0\|_{L^2}$.
 
    Furthermore, under the assumption $\mathcal{E}_u(T) \le 3\mathcal{E}_u(0), \mathcal{E}_d(T)\le 2\delta$ with $\delta$ sufficiently small, we have
    %\eqref{a2} with $\delta$ sufficiently small, we have
	% \begin{equation}\label{basic-est}
	% 	\sup_{0\le t\le T} \left(\bar \rho \| u\|_{L^2}^2+\|\nabla d\|_{L^2}^2\right)   +  \int_{0}^{T} \bar\rho^{\alpha}\|\nabla u\|_{L^2}^2+\|\nabla^2 d\|_{L^2}^2 dt \le C(\bar\rho+\delta).
	% \end{equation}
    \begin{equation}\label{basic-est}
		\sup_{0\le t\le T} \|\nabla d\|_{L^2}^2  +  \int_{0}^{T} \|\nabla^2 d\|_{L^2}^2 dt \le c_2\|\nabla d_0\|_{L^3}^2,
    \end{equation}
    where $c_2$ depends on $\|u_0\|_{H^1},\|\nabla d_0\|_{L^2}$.
\end{lemma}

\begin{proof}
	Multiplying \eqref{ins}$_2$ by $u$ and integrating the resultant equation, we obtain after integration by parts that 
 \begin{equation}\label{d1}
        \frac{1}{2} \frac{d}{dt} \|\sqrt{\rho} u\|_{L^2}^2 + \mu \bar \rho^{\alpha}\|\nabla u\|_{L^2}^2
        \leq \frac{1}{2} \frac{d}{dt} \|\sqrt{\rho} u\|_{L^2}^2+2\mu\int\rho^\alpha|\mathcal{D}(u)|^2=\nu \int\partial_i d\cdot\partial_j d\partial_j u_i dx.
	\end{equation}
 Multiplying \eqref{ins}$_3$ by $- (\Delta d+|\nabla d|^2d)$ and integrating the resultant equation, we obtain from integration by parts and using $|d|=1$ and \eqref{bc}  that 
	\begin{equation}\label{d-1st}
		\begin{aligned}
			&\frac{1}{2} \frac{d}{dt} \|\nabla d\|_{L^2}^2 
   +\lambda \int|\Delta d+|\nabla d|^2d|^2 dx
   =\int u\cdot\nabla d\cdot\Delta d\,dx\\
   =&-\int\partial_j u_i\partial_id\cdot\partial_jd\,dx-\int(u\cdot\nabla)\partial_jd\cdot\partial_j d dx +\int_{\partial\Omega} u\cdot\nabla d\cdot\frac{\partial d}{\partial n}ds\\
   =&-\int\partial_j u_i\partial_id\cdot\partial_jd\,dx.
		\end{aligned}
	\end{equation}
	Adding \eqref{d-1st} multiplied by $\nu$ to \eqref{d1}, integrating over $[0, T]$ and using \eqref{2.3} lead to \eqref{basic-est-0}.
 
 Furthermore, since $|d|=1$ implies $\Delta d\cdot d=-|\nabla d|^2$, one has
 \begin{equation*}
    \int|\Delta d+|\nabla d|^2d|^2\,dx=\int (|\Delta d|^2-|\nabla d|^4)\,dx. 
 \end{equation*}
 % Thus, we obtain from using Lemma \ref{G-N} with $\frac{\partial d}{\partial n}=0$ that
 % \begin{align*}
 %     \|\Delta d\|_{L^2}^2\leq\|\Delta d+|\nabla d|^2d\|_{L^2}^2+\|\nabla d\|_{L^4}^4\leq \|\Delta d+|\nabla d|^2d\|_{L^2}^2+C\|\nabla d\|_{L^3}^2\|\nabla^2 d\|_{L^2}^2
 % \end{align*}
By elliptic estimates Lemma \ref{Nuemann_elliptic} and Sobolev inequality Lemma \ref{G-N}, we have 
\begin{equation}
\begin{aligned}\label{d-ell}
    \|\nabla^2 d\|_{L^2}^2 &\leq C_3^2 \|\Delta d\|_{L^2}^2 \leq C_3^2 \|\Delta d+|\nabla d|^2d\|_{L^2}^2 + C_3^2 \|\nabla d\|_{L^4}^4\\
    &\le C_3^2\|\Delta d+|\nabla d|^2d\|_{L^2}^2 +c_1 \|\nabla d\|_{L^3}^2\|\nabla^2 d\|_{L^2}^2\\
    &\leq C_3^2 \|\Delta d+|\nabla d|^2d\|_{L^2}^2 + 2\delta c_1 \|\nabla^2 d\|_{L^2}^2\\
    &\leq 2C_3^2 \|\Delta d+|\nabla d|^2d\|_{L^2}^2,
\end{aligned}
\end{equation}
after choosing 
$$\delta<\frac{1}{4c_1}.$$
Combining the above result with \eqref{d-1st} gives 
\begin{equation}
    \begin{aligned}
      \frac{d}{dt}\|\nabla d\|_{L^2}^2+\frac{\lambda}{C_3^2}\|\nabla^2 d\|_{L^2}^2
      &\leq C\int|\nabla u||\nabla d|^2 dx
      \leq C\|\nabla u\|_{L^2}\|\nabla d\|_{L^3}\|\nabla d\|_{L^6}\\
      &\leq C\|\nabla u\|_{L^2}\|\nabla d\|_{L^2}^{\frac{1}{2}}\|\nabla^2 d\|_{L^2}^\frac{3}{2}\\
      &\leq \frac{\lambda}{2C_3^2}\|\nabla^2 d\|_{L^2}^2+C\|\nabla u\|_{L^2}^4\|\nabla d\|_{L^2}^2,
    \end{aligned}
\end{equation}
where we have used 
\begin{equation*}
    \|\nabla d\|_{L^3}\leq C\|\nabla d\|_{L^2}^\frac{1}{2}\|\nabla
    ^2d\|_{L^2}^\frac{1}{2},~~\|\nabla d\|_{L^6}\leq C\|\nabla^2d\|_{L^2},
\end{equation*}
which follows from using \eqref{G-N} with $\frac{\partial d}{\partial n}=0$ on $\partial\Omega$.
Applying Gronwall's inequality, we obtain
\begin{equation}\label{359}
    \begin{aligned}
    &\sup_{t\in[0,T]}\|\nabla d\|_{L^2}^2 
    +\frac{\lambda}{C_3^2} \int_{0}^{T} \|\nabla^2 d\|_{L^2}^2 dt \leq \|\nabla d_0\|_{L^2}^2\cdot\exp\left\{C\int_0^T\|\nabla u\|_{L^2}^4 dt\right\}\\
    %&\leq \|\nabla d_0\|_{L^2}^2\exp\left\{\Tilde C(\|\nabla u_0\|_{L^2}, \|\nabla^2 d_0\|_{L^2})\int_0^T\|\nabla u\|_{L^2}^2 dt\right\}\\
    %&\leq C\|\nabla d_0\|_{L^2}^2 \exp\left\{C(\bar\rho^{1-\alpha}\|\nabla u_0\|_{L^2}^2+\bar\rho^{1-\alpha}\|\nabla^2 d_0\|_{L^2}^2)\right\}\\
    &\leq \|\nabla d_0\|_{L^2}^2
    \exp\left\{3C\mathcal{E}_u(0)\bar\rho^{1-2\alpha}\right\}\leq \|\nabla d_0\|_{L^2}^2
    \exp\left\{3C\|\nabla u_0\|_{L^2}^2\bar\rho^{1-\alpha}\right\}\\
    &\leq C\|\nabla d_0\|_{L^2}^2
    \leq \Tilde{c}_2 \|\nabla d_0\|_{L^3}^2, 
    \end{aligned}
\end{equation}
which together with $c_2\triangleq \max\{\Tilde{c}_2, \lambda^{-1}C^2_3\Tilde{c}_2\}$ yields \eqref{basic-est}.
%Combining the above inquality with \eqref{basic-est-0} and $\|\nabla d_0\|_{L^2}^2\leq C\|\nabla d_0\|_{L^3}^2$, yields \eqref{basic-est}.
\end{proof}

Now we can close the a prior assumption $\mathcal{E}_d(T)$.
\begin{lemma}\label{3d}
Under the assumption $\mathcal{E}_u(T) \le 3\mathcal{E}_u(0), \mathcal{E}_d(T)\le 2\delta$ with $\delta$ sufficiently small, 
we have
\begin{align}\label{d2d}
  \sup_{0\le t\le T}\|\nabla^2 d\|_{L^2}^2 
    +\int_{0}^{T}\|\nabla^3 d\|_{L^2}^2 dt\leq C\|\nabla^2d_0\|_{L^2}^2,
\end{align}
and
\begin{align}\label{tdd}
 \sup_{0\le t\le T}t\|\nabla^2 d\|_{L^2}^2 
    +\int_{0}^{T}t\|\nabla^3 d\|_{L^2}^2 dt\leq C\|\nabla d_0\|_{L^3}^2.   
\end{align}
As a consequence, there exists a positive constant $\varepsilon_0$ such that 
\begin{equation}
\mathcal{E}_d(T)\le \delta,
\end{equation}
provided $\|\nabla d_0\|_{L^3}\leq  \varepsilon_0=\varepsilon_0(\Omega,C_0,\mu, \nu, \lambda, \alpha, \|\nabla \rho_0\|_{L^q}, \| u_0\|_{H^1}, \|\nabla^2 d_0\|_{L^2})$.
\end{lemma}
\begin{proof}
Taking the gradient of $\eqref{ins}_3$, we get
that
\begin{equation}\label{dd}
    \nabla d_t- \lambda \nabla\Delta d=-\nabla(u\cdot\nabla d)+ \lambda \nabla(\vert\nabla d\vert^2 d).
\end{equation}
Multiplying \eqref{dd} by $-\nabla\Delta d $, and using integration by parts, we have 
\begin{equation}\label{d-2nd}
    \begin{aligned}
    &\frac{1}{2}\frac{d}{dt}\|\Delta d\|_{L^2}^2+\lambda \|\nabla\Delta d\|_{L^2}^2=\int\nabla(u\cdot\nabla d)\cdot\nabla \Delta d\,dx-\lambda \int\nabla(|\nabla d|^2d)\nabla\Delta d\,dx\\
    &\leq (\|\nabla u\|_{L^2}\|\nabla d\|_{L^\infty}+\|u\|_{L^6}\|\nabla^2d\|_{L^3}+\|\nabla d\|_{L^3}\|\nabla^2d\|_{L^6}+\|\nabla d\|_{L^\infty}^{\frac{3}{2}}\|\nabla d\|_{L^3}^{\frac{3}{2}})\|\nabla\Delta d\|_{L^2}\\
    &\leq C \|\nabla u\|_{L^2} \|\nabla d\|_{L^3}^{\frac{1}{3}} \|\nabla^3 d\|_{L^2}^{\frac{5}{3}} + C(\|\nabla d\|_{L^3}+\|\nabla d\|_{L^3}^2)\|\nabla^3d\|_{L^2}^2\\
    &\leq C(\eta+\|\nabla d\|_{L^3}+\|\nabla d\|_{L^3}^2)\|\nabla^3d\|_{L^2}^2+C_\eta\|\nabla d\|_{L^3}^2 \|\nabla u\|_{L^2}^6,
\end{aligned}
\end{equation}
where we have used 
% $$\|\nabla d\|_{L^\infty} \le C \|\nabla^2 d\|_{L^3}\leq C\|\nabla^3 d\|_{L^2}^{\frac{1}{2}}\|\nabla^2 d\|_{L^2}^\frac{1}{2}+ C\|\nabla^2 d\|_{L^2} ,$$
% and 
\begin{equation}\label{GN-d-2}
   \|\nabla^2 d\|_{L^3}\leq C\|\nabla^3 d\|_{L^2}^{\frac{2}{3}}\|\nabla d\|_{L^3}^\frac{1}{3},~~\|\nabla d\|_{L^\infty}\leq C\|\nabla d\|_{L^3}^\frac{1}{3}\|\nabla^3 d\|_{L^2}^\frac{2}{3}, 
\end{equation}
and
\begin{equation*}
-\int\partial_t\nabla d\cdot\nabla\Delta ddx=-\int_{\partial\Omega}\partial_t\frac{\partial d}{\partial n}\cdot\Delta d ds
+\int\partial_t\Delta d\cdot\Delta d dx=\frac{1}{2}\frac{d}{dt}\|\Delta d\|_{L^2}^2.
\end{equation*}
Then using elliptic estimates (Lemma \ref{slip_estimate}) and choosing $\eta$ small enough, we have
\begin{equation}\label{d-2nd1}
    \begin{aligned}
    \frac{d}{dt}\|\Delta d\|_{L^2}^2
    +\lambda\|\nabla^3 d\|_{L^2}^2
    \leq& c_3 \delta \|\nabla^3 d\|_{L^2}^2+C\|\nabla d\|_{L^3}^2 \|\nabla u\|_{L^2}^6\\
    \leq& c_3 \delta \|\nabla^3 d\|_{L^2}^2+C\|\nabla d\|_{L^2}\|\nabla^2 d\|_{L^2} \|\nabla u\|_{L^2}^6\\
    \leq& c_3 \delta \|\nabla^3 d\|_{L^2}^2+C\|\nabla^2 d\|_{L^2}^2 \|\nabla u\|_{L^2}^6\\
    \leq& c_3 \delta \|\nabla^3 d\|_{L^2}^2+C\|\Delta d\|_{L^2}^2 \|\nabla u\|_{L^2}^6,
\end{aligned}
\end{equation}
where we have used the Poincare's inequality. 
Then after choosing 
$$\delta<\frac{\lambda}{2c_3},$$
Gronwall's inequality implies 
\begin{equation}\label{d-2nd2}
    \begin{aligned}
    &\sup_{0\le t\le T}\|\nabla^2 d\|_{L^2}^2 
    +\frac{\lambda}{2}\int_{0}^{T}\|\nabla^3 d\|_{L^2}^2 dt
    \leq C \|\nabla^2 d_0\|_{L^2}^2 \cdot \exp{\left\{C\int_{0}^{T}\|\nabla u\|_{L^2}^6dt\right\}} \\
    \le& C \|\nabla^2 d_0\|_{L^2}^2 \cdot \exp\{C \bar\rho^{1-\alpha}\} \le c_4 \|\nabla^2 d_0\|_{L^2}^2 .
\end{aligned}
\end{equation}
Multiplying \eqref{d-2nd1} by $t$, integrating the resultant equation and using \eqref{basic-est} give
\begin{equation}\label{d-2nd2-t}
    \begin{aligned}
    &\sup_{0\le t\le T} t\|\nabla^2 d\|_{L^2}^2 
    +\frac{\Tilde{C}}{2} \int_{0}^{T} t \|\nabla^3 d\|_{L^2}^2 dt
    \leq C  \int_{0}^{T}\|\Delta d\|_{L^2}^2 dt \cdot \exp{\left\{C\int_{0}^{T}\|\nabla u\|_{L^2}^6 dt\right\}} \\
    \le& C c_2 \|\nabla d_0\|_{L^3}^2 \exp\{C \bar\rho^{1-\alpha}\} \le C \|\nabla d_0\|_{L^3}^2.
\end{aligned}
\end{equation}
% Multiplying \eqref{ins}$_3$ by $-\Delta d$, integrating the resultant equation and using the similar argument as \eqref{d-1st}-\eqref{d-ell}, we have
% \begin{equation}
%     \begin{aligned}
%       \frac{d}{dt}\|\nabla d\|_{L^2}^2+\|\nabla^2 d\|_{L^2}^2
%       &\leq C\int|\nabla u||\nabla d|^2 dx
%       \leq C\|\nabla u\|_{L^2}\|\nabla d\|_{L^3}\|\nabla d\|_{L^6}\\
%       &\leq C\|\nabla u\|_{L^2}\|\nabla d\|_{L^2}^{\frac{1}{2}}\|\nabla^2 d\|_{L^2}^\frac{3}{2}\\
%       &\leq \frac{1}{2}\|\nabla^2 d\|_{L^2}^2+C\|\nabla u\|_{L^2}^4\|\nabla d\|_{L^2}^2.
%     \end{aligned}
% \end{equation}
% Applying Gronwall's inequality, we obtain
% \begin{equation}\label{359}
%     \begin{aligned}
%     &\sup_{t\in[0,T]}\|\nabla d\|_{L^2}^2 + \int_{0}^{T} \|\nabla^2 d\|_{L^2}^2 dt \leq \|\nabla d_0\|_{L^2}^2\cdot\exp\left\{C\int_0^T\|\nabla u\|_{L^2}^4 dt\right\}\\
%     %&\leq \|\nabla d_0\|_{L^2}^2\exp\left\{\Tilde C(\|\nabla u_0\|_{L^2}, \|\nabla^2 d_0\|_{L^2})\int_0^T\|\nabla u\|_{L^2}^2 dt\right\}\\
%     %&\leq C\|\nabla d_0\|_{L^2}^2 \exp\left\{C(\bar\rho^{1-\alpha}\|\nabla u_0\|_{L^2}^2+\bar\rho^{1-\alpha}\|\nabla^2 d_0\|_{L^2}^2)\right\}\\
%     &\leq C\|\nabla d_0\|_{L^2}^2
%     \exp\left\{C\bar\rho^{1-\alpha}\right\}\\
%     &\leq C\|\nabla d_0\|_{L^2}^2
%     \leq C_4\|\nabla d_0\|_{L^3}^2.  
%     \end{aligned}
% \end{equation}
Finally, \eqref{359} together with \eqref{a1} yields
\begin{equation}
\begin{aligned}
    \mathcal{E}_d(T)&\leq C \sup_{t\in[0,T]}\|\nabla d\|_{L^2}\sup_{t\in[0,T]}\|\nabla^2 d\|_{L^2}\\
    &\leq C \sqrt{c_4} \|\nabla^2 d_0\|_{L^2}\sqrt{c_2}\|\nabla d_0\|_{L^3}<\delta,
\end{aligned}
\end{equation}
provided 
$$\|\nabla d_0\|_{L^3} < \varepsilon_0 \triangleq \frac{\delta}{C \sqrt{c_2 c_4}\|\nabla^2 d_0\|_{L^2}}.$$
\end{proof}

High-order a priori estimates rely on the following regularity results for density-dependent Stokes equations.
\begin{lemma}\label{stokes-e}
Assume that $\rho\in W^{1,q}, 3<q<6$, and $\bar \rho \le \rho \le C_0 \bar \rho$.
Let $(u, P) \in H_{0,\sigma}^1\times L^2$ be the unique weak solution to the boundary value problem
\begin{equation}\label{stokes}
	\left\{ \begin{array}{l}
		-\mathrm{div}(2\mu\rho^\alpha D(u))+\nabla P=F,~~\mathrm{in}~\Omega, \\
		\mathrm{div} u=0,~~\mathrm{in}~\Omega,\\
        \int \frac{P}{\rho^\alpha} dx=0,~~\mathrm{in}~\Omega.
	\end{array} \right.
\end{equation}
Then we have the following regularity results:

(1) If $F\in L^2$, then $(u, P)\in H^2\times H^1$ and
\begin{equation}\label{stokes-2}
\| u\|_{H^2}+\bigg\|\frac{P}{\rho^\alpha}\bigg\|_{H^1}
\leq C(\bar\rho^{-\alpha}+\bar\rho^{-\alpha-\frac{q}{q-3}}\|\nabla \rho \|_{L^q}^\frac{q}{q-3})\|F\|_{L^2};
\end{equation}

(2) If $F\in L^q$ for some $q\in (3,6)$ then $(u, P)\in W^{2,q}\times W^{1,q}$ and
\begin{equation}\label{stokes-q}
\| u\|_{W^{2,q}}+\bigg\|\frac{P}{\rho^\alpha}\bigg\|_{W^{1,q}}
\leq  C(\bar\rho^{-\alpha}+\bar\rho^{-\alpha-\frac{5q-6}{2(q-3)}}\|\nabla \rho \|_{L^q}^\frac{5q-6}{2(q-3)})\|F\|_{L^q}.
\end{equation}
Here the constant C in \eqref{stokes-2} and \eqref{stokes-q} depends on $\Omega, q$.
\end{lemma}
\begin{proof}
The proof of this theorem is the same as in \cite{huang2024global}, we repeat it here for the reader's convenience.
Multiply the first equation of $\eqref{stokes}_1$ by $u$ and integrate over $\Omega$, then by Cauchy's inequality,
\begin{equation}
    \int2\rho^\alpha\vert D(u)\vert^2 dx=\int F\cdot udx\leq \|F\|_{L^2}\|u\|_{L^2}.
\end{equation}
Note that
\begin{equation}\label{deg}
    2\int \vert D(u)\vert^2dx=\int\vert\nabla u\vert^2dx,
\end{equation}
hence, it follows from \eqref{2.3} and \eqref{basic-est} that
\begin{equation}
    \|\nabla u\|_{L^2}\leq C\bar{\rho}^{-\alpha}\| u\|_{L^2}\| F\|_{L^2}\leq C\bar{\rho}^{-\alpha}\| F\|_{L^2}.
\end{equation}

Since $\int\frac{P}{\rho^\alpha}dx =0$, according to Lemma \ref{bovosgii}, there exists a function $v\in H^1_0$, such that
\begin{equation}
    \mathrm{div}v=\frac{P}{\rho^\alpha},
\end{equation}
and
\begin{equation}
    \|\nabla v\|_{L^2}\leq C\bigg\|\frac{P}{\rho^\alpha}\bigg\|_{L^2}.
\end{equation}
Multiplying the first equation of \eqref{stokes} by $-v$, and integrating over $\Omega$, then making use of Poincar\'e's inequality, one obtains
\begin{equation}
\begin{split}
    \int\frac{P^2}{\rho^\alpha}dx
    &=-\int F\cdot vdx+2\int \mu\rho^\alpha D(u):\nabla v dx\\
    &\leq C\|F\|_{L^2}\|v\|_{L^2}
    +C\bar{\rho}^\alpha\|\nabla u\|_{L^2}\|\nabla v\|_{L^2}\\
    &\leq C\|F\|_{L^2}\|\nabla v\|_{L^2}\\
    &\leq C\|F\|_{L^2}\bigg\|\frac{P}{\rho^\alpha}\bigg\|_{L^2}.
\end{split}
\end{equation}
On the other hand,
\begin{equation}
    \int\frac{P^2}{\rho^\alpha}dx\geq C\bar{\rho} \int\frac{P^2}{\rho^{2\alpha}}dx,
\end{equation}
hence
\begin{equation}
    \bigg\|\frac{P}{\rho^\alpha}\bigg\|_{L^2}\leq C\bar{\rho}^{-\alpha}\|F\|_{L^2}.
\end{equation}

Rewrite \eqref{stokes}$_2$ as 
\begin{equation}
	-\mu\Delta u+\nabla\biggl(\frac{P}{\rho^\alpha}\biggr)=\frac{F}{\rho^\alpha}+2\mu\alpha \rho^{-1}\nabla \rho \cdot D(u)
 +\alpha \rho^{-1}\nabla \rho \frac{P}{\rho^\alpha}. 
\end{equation}
Stokes estimates and Lemma \ref{G-N} yield
\begin{equation}
	\begin{aligned}
		&\|\nabla^2 u\|_{L^2}+\bigg\|\nabla\biggl(\frac{P}{\rho^\alpha}\biggr)\bigg\|_{L^2} \\	
		\leq& C\biggl(\bar\rho^{-\alpha}\|F\|_{L^2}
        +\bar\rho^{-1} \|\nabla\rho\cdot\nabla u\|_{L^2}
        +\bar\rho^{-1} \|\nabla\rho\cdot\frac{P}{\rho^\alpha}\|_{L^2}\biggr)  \\
        \leq& C\biggl(\bar\rho^{-\alpha}\|F\|_{L^2}
        +\bar\rho^{-1} \|\nabla \rho \|_{L^q} \|\nabla u\|_{L^{\frac{2q}{q-2}}}
        +\bar\rho^{-1} \|\nabla \rho \|_{L^q} \bigg\|\frac{P}{\rho^\alpha}\bigg\|_{L^{\frac{2q}{q-2}}} \biggr) \\
		\leq& C\biggl(\bar\rho^{-\alpha}\|F\|_{L^2}
        +\bar\rho^{-1} \|\nabla \rho \|_{L^q} \|\nabla u\|_{L^2}^\frac{q-3}{q} \|\nabla^2 u\|_{L^2}^\frac{3}{q}
        +\bar\rho^{-1} \|\nabla \rho \|_{L^q} \bigg\|\frac{P}{\rho^\alpha}\bigg\|_{L^2}^\frac{q-3}{q}\bigg\|\nabla\biggl(\frac{P}{\rho^\alpha}\biggr)\bigg\|_{L^2}^\frac{3}{q} \biggr). 
	\end{aligned}
\end{equation}
By Young’s inequality,
\begin{equation}
	\begin{aligned}
		&\|\nabla^2 u\|_{L^2}+\bigg\|\nabla\biggl(\frac{P}{\rho^\alpha}\biggr)\bigg\|_{L^2} \\	
		\leq& C\bar\rho^{-\alpha}\|F\|_{L^2}
        +C\bar\rho^{-\frac{q}{q-3}} \|\nabla \rho \|_{L^q}^\frac{q}{q-3} \biggl(\|\nabla u\|_{L^2}+\bigg\|\frac{P}{\rho^\alpha}\bigg\|_{L^2} \biggr)  \\
        \leq& C(\bar\rho^{-\alpha}+\bar\rho^{-\alpha-\frac{q}{q-3}}\|\nabla \rho \|_{L^q}^\frac{q}{q-3})\|F\|_{L^2}.
	\end{aligned}
\end{equation}
Similarly,
\begin{equation}
\|\nabla^2 u\|_{L^q}+\bigg\|\nabla\biggl(\frac{P}{\rho^\alpha}\biggr)\bigg\|_{L^q}
        \leq  C(\bar\rho^{-\alpha}+\bar\rho^{-\alpha-\frac{5q-6}{2(q-3)}}\|\nabla \rho \|_{L^q}^\frac{5q-6}{2(q-3)})\|F\|_{L^q}.
\end{equation}

\end{proof}

% \begin{equation}
% 	\begin{aligned}
% 		&\|\nabla^2 u\|_{L^2}+\bigg\|\nabla\biggl(\frac{P}{\rho^\alpha}\biggr)\bigg\|_{L^2} \\	
% 		\le& C \bar\rho^{-\alpha+\frac12}\|\sqrt{\rho}u_t\|_{L^2}+ \bar\rho^{-\alpha+1}  \|u\cdot \nabla u\|_{L^2} + \bar\rho^{-1} \|\nabla\rho\cdot\nabla u\|_{L^2}
%         +\bar\rho^{-1} \|\nabla\rho\cdot\frac{P}{\rho^\alpha}\|_{L^2}  \\
% 		\le & C \bar\rho^{-\alpha+\frac12}\|\sqrt{\rho}u_t\|_{L^2}+ \bar\rho^{-\alpha+1} \|\nabla u\|_{L^2}^{\frac{3}{2}} \|\nabla u\|_{L^6}^{\frac{1}{2}} + \bar\rho^{-1} \|\nabla \rho \|_{L^q} \|\nabla u\|_{L^{\frac{2q}{q-2}}}\\
% 		\le & \frac{1}{2} \|\nabla u\|_{L^6} + C \bar\rho^{-\alpha+\frac12}\|\sqrt{\rho}u_t\|_{L^2}+ \bar\rho^{-2\alpha+2} \|\nabla u\|_{L^2}^{3}  + \bar\rho^{-\frac{q}{q-3}} \|\nabla \rho \|_{L^q}^{\frac{q}{q-3}} \|\nabla u\|_{L^{2}},
% 	\end{aligned}
% \end{equation}
% and 
% \begin{equation}
% 	\begin{aligned}
% 		&\|\nabla^2 u\|_{L^q} \le C \|H\|_{L^q}\\
% 		\le& C \bar\rho^{-\alpha}\|-\rho(u_t+u\cdot \nabla u)+2\mu \nabla \rho^\alpha \cdot \mathcal{D}u\|_{L^q} \\
% 		\le& C \bar\rho^{-\alpha}\|{\rho}u_t\|_{L^q}+ \bar\rho^{-\alpha+1}  \|u\cdot \nabla u\|_{L^q} + \bar\rho^{-1} \|\nabla \rho \mathcal{D}u\|_{L^q}\\
% 		\le &C \left(\bar\rho^{-\alpha+\frac{5q-6}{4q}}\|\sqrt{\rho}u_t\|_{L^2}^{\frac{6-q}{2q}} \|\nabla u_t\|_{L^2}^{\frac{3(q-3)}{2q}} + \bar\rho^{-\alpha+1}  \| u\|_{L^q} \| \nabla u\|_{L^\infty} + \bar\rho^{-1} \|\nabla \rho\|_{L^q} \| \nabla u\|_{L^\infty}\right),
% 	\end{aligned}
% \end{equation}

As a consequence, we have the following high-order estimate of the velocities which will be used frequently.

\begin{lemma}
	Under the assumption \eqref{a1}, it holds that 
	\begin{equation}\label{H2}
	    \| u\|_{H^2} \leq C(\bar\rho^{\frac{1}{2}-\alpha}\|\sqrt{\rho} u_t\|_{L^2}
+\bar\rho^{2-2\alpha}\|\nabla u\|_{L^2}^3+\bar\rho^{-\alpha}\|\nabla d\|_{L^3}\|\nabla^3d\|_{L^2}),
	\end{equation}
		and
        \begin{equation}\label{W2q}
	    \| u\|_{W^{2,q}} \leq C \bar\rho^{-\alpha}  \|\rho u_t\|_{L^q} + C \bar \rho^{(1-\alpha)\frac{5q-6}{q}} \|\nabla u\|_{L^2}^{\frac{6(q-1)}{q}}+ C \bar\rho^{-\alpha} \|\nabla d\|_{L^3}^{\frac{2}{q}}\|\nabla^3 d\|_{L^2}^{2-\frac{2}{q}}.
	\end{equation}	
\end{lemma}
\begin{proof}
Let 
\begin{equation}
    F=-\rho \dot u-\nu \mathrm{div}(\nabla d\odot\nabla d)
\end{equation}
in Lemma \ref{stokes-e}. 
\eqref{stokes-2} together with \eqref{a1}, \eqref{basic-est} and \eqref{GN-p} gives
\begin{equation}\label{d2uu}
\begin{aligned}
\|u\|_{H^2}
\leq& C(\bar\rho^{-\alpha}+\bar\rho^{-\alpha-\frac{q}{q-3}}\mathcal{E}_\rho(0)^\frac{q}{q-3})(\|\rho \dot  u\|_{L^2}+\|\mathrm{div}(\nabla d\odot\nabla d)\|_{L^2})\\
\leq& C\bar\rho^{-\alpha}(\|\rho \dot  u\|_{L^2}+\|\mathrm{div}(\nabla d\odot\nabla d)\|_{L^2})\\
\leq& C\bar\rho^{\frac{1}{2}-\alpha}\|\sqrt\rho u_t\|_{L^2} +C \bar\rho^{1-\alpha} \|u\|_{L^6}\|\nabla u\|_{L^3} +C \bar\rho^{-\alpha}\|\nabla d\|_{L^3}\|\nabla^2 d\|_{L^6}\\
\leq& C \bar\rho^{\frac{1}{2}-\alpha}\|\sqrt\rho u_t\|_{L^2} +C \bar\rho^{1-\alpha}\|\nabla u\|_{L^2}^\frac{3}{2}\|\nabla u\|_{H^1}^\frac{1}{2}+C\bar\rho^{-\alpha}\|\nabla d\|_{L^3}\|\nabla^3 d\|_{L^2}\\
\leq& \frac 12 \|\nabla u\|_{H^1}+  C(\bar\rho^{\frac{1}{2}-\alpha}\|\sqrt{\rho} u_t\|_{L^2}
+\bar\rho^{2-2\alpha}\|\nabla u\|_{L^2}^3+\bar\rho^{-\alpha}\|\nabla d\|_{L^3}\|\nabla^3 d\|_{L^2}),
\end{aligned}
\end{equation}
provided $\bar\rho\geq 1$. Here we have used $\|\nabla^2d\|_{L^6}\leq C\|\nabla^3 d\|_{L^2}$ which follows from \eqref{GN-p} with $\frac{\partial d}{\partial n}\big|_{\partial\Omega}=0.$
Similarly, Gagliardo-Nirenber inequality together with \eqref{stokes-q} yields
    \begin{equation}
		\begin{aligned} 
		\| u\|_{W^{2,q}} \le & C \bar\rho^{-\alpha}  (\|\rho \dot u\|_{L^q}+\|\mathrm{div}(\nabla d\odot\nabla d)\|_{L^q}) \\
        \le & C \bar\rho^{-\alpha} (  \|\rho u_t\|_{L^q} + \bar \rho \|u\|_{L^6} \| \nabla u\|_{L^{\frac{6q}{6-q}}}+\|\nabla d\|_{L^{\frac{6q}{6-q}}}\|\nabla^2d\|_{L^6})\\
        \le & C \bar\rho^{-\alpha} (  \|\rho u_t\|_{L^q} + \bar \rho \|\nabla u\|_{L^2}^{\frac{6(q-1)}{5q-6}} \|\nabla  u\|_{W^{1,q}}^{\frac{4q-6}{5q-6}}+\|\nabla d\|_{L^3}^{\frac{2}{q}}\|\nabla^2 d\|_{L^6}^{2-\frac{2}{q}})\\
        \leq & \frac12 \| \nabla u\|_{W^{1,q}} + C \bar\rho^{-\alpha}  \|\rho u_t\|_{L^q} + C \bar \rho^{(1-\alpha)\frac{5q-6}{q}} \|\nabla u\|_{L^2}^{\frac{6(q-1)}{q}}+ C \bar\rho^{-\alpha} \|\nabla d\|_{L^3}^{\frac{2}{q}}\|\nabla^3 d\|_{L^2}^{2-\frac{2}{q}} \\
        \leq & C \bar\rho^{-\alpha}  \|\rho u_t\|_{L^q} + C \bar \rho^{(1-\alpha)\frac{5q-6}{q}} \|\nabla u\|_{L^2}^{\frac{6(q-1)}{q}}+ C \bar\rho^{-\alpha} \|\nabla d\|_{L^3}^{\frac{2}{q}}\|\nabla^3 d\|_{L^2}^{2-\frac{2}{q}}.
		\end{aligned}
	\end{equation}    
\end{proof}

Now we are ready to deal with an estimate to $\mathcal{E}_u(T)$.
\begin{lemma}\label{L_2}
There exists a positive constant $\Lambda_1$ such that 
\begin{equation}
\mathcal{E}_u(T)\le 2\mathcal{E}_u(0),
\end{equation}
and
\begin{equation}\label{tdu}
\sup_{t\in[0,T]}t\bar\rho^{\alpha}\|\nabla u \|_{L^2}^2
+\int_0^{T}t \|\sqrt{\rho} u_t \|_{L^2}^2 dt \leq C\bar\rho ,
\end{equation}
% \begin{equation}\label{tdd}
% \sup_{t\in[0,T]}t \|\nabla^2d\|_{L^2}^2
% +\int_0^{T}t \|\nabla^3d\|_{L^2}^2 dt \leq C \|\nabla d_0\|_{L^3}^2,
% \end{equation}
provided $\bar \rho\ge \Lambda_1= \Lambda_1(\Omega,C_0,\mu,\nu, \lambda, \alpha, \|\nabla \rho_0\|_{L^q}, \| u_0\|_{H^1}, \|\nabla^2 d_0\|_{L^2})$.
\end{lemma}
\begin{proof}
	Multiplying \eqref{ins}$_2$ by $u_t$, and integrating by parts, we have that 
	\begin{equation}\label{51}
		\begin{aligned}
			&\frac{d}{dt}\int \mu\rho^\alpha |\mathcal{D}(u)|^2dx+  \int\rho |u_t|^2dx \\
			&= -\int  \rho u \cdot \nabla u \cdot u_t dx+ \int  \mu(\rho^\alpha)_t |\mathcal{D}(u)|^2 dx-\nu \int\mathrm{div}(\nabla d\odot\nabla d)\cdot u_t dx.
		\end{aligned}
	\end{equation}
	It follows from H\"older and Sobolev inequalities that 
	\begin{equation}\label{52}
        \begin{aligned}
		&\int  \rho u \cdot \nabla u \cdot u_t dx
        \leq C\bar\rho^\frac{1}{2}\|\sqrt{\rho} u_t\|_{L^2}\|u\|_{L^6}\|\nabla u\|_{L^3} \\
        &\leq C\bar\rho^\frac{1}{2}\|\sqrt{\rho} u_t\|_{L^2}\|\nabla u\|_{L^2}^\frac{3}{2}\|\nabla u\|_{H^1}^\frac{1}{2}\\
        &\leq C\bar\rho^\frac{1}{2}\|\sqrt{\rho} u_t\|_{L^2}\|\nabla u\|_{L^2}^\frac{3}{2}(\bar\rho^{\frac{1}{2}-\alpha}\|\sqrt{\rho} u_t\|_{L^2}
        +\bar\rho^{2-2\alpha}\|\nabla u\|_{L^2}^3+\bar\rho^{-\alpha}\|\nabla d\|_{L^3}\|\nabla^3d\|_{L^2})^\frac{1}{2}\\
        &\leq \frac{1}{16}\|\sqrt{\rho} u_t\|_{L^2}^2
        +C(\bar\rho^{3-2\alpha}+\bar\rho^{2-2\alpha})\|\nabla u\|_{L^2}^6+C\|\nabla d\|_{L^3}^2\|\nabla^3d\|_{L^2}^2.
        \end{aligned}
	\end{equation}
Using the fact that
\begin{equation}\label{rhoat}
    \partial_t(\rho^\alpha)+u\cdot\nabla\rho^\alpha=0,
\end{equation}
due to \eqref{ins}$_1$ and \eqref{ins}$_4$, which together with \eqref{d2uu} yields
\begin{equation}\label{53}
\begin{aligned}
&\int\mu(\rho^\alpha)_t |\mathcal{D}(u)|^2 dx\le C \bar \rho^{\alpha-1} \int  \left|\nabla \rho \cdot u\right||\nabla u|^2 dx\\
&\le C \bar \rho^{\alpha-1} \|\nabla \rho \|_{L^q}\|u\|_{L^6} \|\nabla u\|_{L^{\frac{12q}{5q-6}}}^2\\
&\le C\bar{\rho}^{\alpha-1} \|\nabla \rho \|_{L^q}\|\nabla u\|_{L^2}^{\frac{5q-6}{2q}} \|\nabla u\|_{H^1}^{\frac{q+6}{2q
}}\\
&\le C\mathcal{E}_\rho(0)\bar{\rho}^{\alpha-1} \|\nabla u\|_{L^2}^{\frac{5q-6}{2q}} (\bar\rho^{\frac{1}{2}-\alpha} \|\sqrt{\rho} u_t\|_{L^2
}
+\bar\rho^{2-2\alpha}\|\nabla u\|_{L^2}^3+\bar\rho^{-\alpha}\|\nabla d\|_{L^3}\|\nabla^3d\|_{L^2} )^{\frac{q+6}{2q
}}\\
&\leq \frac{1}{16} \|\sqrt{\rho} u_t\|_{L^2}^2 + C \bar \rho^{\frac{6}{q}(1-\alpha)} \|\nabla u\|_{L^2}^{4+\frac{6}{q
}}\\
&+C\left(\bar\rho^{-\frac{q+6}{3q-6}+\frac{12-2q}{3q-6}(1-\alpha)}+\bar\rho^{-\frac{2q+12}{3q-6}+\frac{12-2q}{3q-6}(1-\alpha)}\right)
\|\nabla u\|_{L^2}^\frac{2(5q-6)}{3(q-2)
}
+C\|\nabla d\|_{L^3}^2\|\nabla^3d\|_{L^2}^2.
\end{aligned}
\end{equation}

Using integration by parts and $\eqref{ins}_3$ yield that
\begin{equation}
\begin{aligned}
    &-\int\mathrm{div}(\nabla d\odot\nabla d)\cdot u_t dx=\int \nabla d\odot\nabla d:\nabla u_t dx
\\
    &=\frac{d}{dt}\int\nabla d\odot\nabla d:\nabla u dx-2\int\nabla \partial_td\odot\nabla d:\mathcal{D}(u) dx
\\
    &\leq \frac{d}{dt}\int\nabla d\odot\nabla d:\nabla u dx+\int|\nabla u|^2|\nabla d|^2 dx+\int|u||\nabla^2 d||\nabla d||\nabla u| dx
\\
    &\quad +\int|\nabla\Delta d||\nabla d||\nabla u| dx+\int|\nabla d|^2|\nabla^2d||\nabla u| dx+\int|\nabla d|^4|\nabla u| dx\\
    &=\frac{d}{dt}\int\nabla d\odot\nabla d:\nabla u dx+\sum_{i=1}^5I_i.
\end{aligned}    
\end{equation}
It follows from H\"older's inequality, Gagliardo-Nirenberg inequality, and \eqref{H2} that
\begin{align}\label{I-1}
    I_1&\leq \|\nabla d\|_{L^3}^2\|\nabla u\|_{L^6}^2\leq C\|\nabla d\|_{L^3}^2\|\nabla u\|_{H^1}^2\nonumber\\
    &\leq C\bar\rho^{1-2\alpha}\|\nabla d\|_{L^3}^2\|\sqrt{\rho} u_t\|_{L^2}^2+C(\|\nabla d\|_{L^3}^2\bar\rho^{4-4\alpha}\|\nabla u\|_{L^2}^6+\bar\rho^{-2\alpha}\|\nabla d\|_{L^3}^4\|\nabla^3d\|_{L^2}^2),
\end{align}
\begin{equation}
    \begin{aligned}
        I_2&\leq \|u\|_{L^6}\|\nabla^2 d\|_{L^3}\|\nabla d\|_{L^3}\|\nabla u\|_{L^6}
        \leq C\|\nabla u\|_{L^2}\|\nabla^3 d\|_{L^2}^{\frac{2}{3}}\|\nabla d\|_{L^3}^\frac{4}{3}\|\nabla u\|_{H^1}\\
        &\leq \frac{1}{16}\|\sqrt{\rho}u_t\|_{L^2}^2 
        +C\|\nabla d\|_{L^3}^2\|\nabla^3d\|_{L^2}^2
        +(\bar\rho^{3-3\alpha}\|\nabla d\|_{L^3} +\bar\rho^{3-6\alpha}\|\nabla d\|_{L^3}^4)\|\nabla u\|_{L^2}^6,
    \end{aligned}
\end{equation} 
\begin{equation}
   \begin{aligned}
        I_3&\leq \|\nabla\Delta d\|_{L^2}\|\nabla d\|_{L^3}\|\nabla u\|_{L^6}
        \leq C\|\nabla^3 d\|_{L^2}\|\nabla d\|_{L^3}\|\nabla u\|_{H^1}\\
        &\leq \frac{1}{16}\|\sqrt{\rho}u_t\|_{L^2}^2
        +C\bar\rho^{-\alpha}\|\nabla d\|_{L^3}^2\|\nabla^3d\|_{L^2}^2
        +C\bar\rho^{4-3\alpha}\|\nabla u\|_{L^2}^6.
   \end{aligned}
\end{equation}
Similarly, 
\begin{equation}
\begin{aligned}
    I_4&\leq \|\nabla d\|_{L^3}^2\|\nabla^2 d\|_{L^6}\|\nabla u\|_{L^6}
    \leq C\|\nabla d\|_{L^3}^2\|\nabla^3 d\|_{L^2}\|\nabla u\|_{H^1}\\
    &\leq \frac{1}{16}\|\sqrt{\rho}u_t\|_{L^2}^2
    +C\bar\rho^{-\alpha}\|\nabla d\|_{L^3}^2\|\nabla^3d\|_{L^2}^2+C\bar\rho^{4-3\alpha}\|\nabla u\|_{L^2}^6,
\end{aligned}
\end{equation}
and
\begin{equation}\label{I-5}
    \begin{aligned}
        I_5&\leq \|\nabla d\|_{L^\infty}^\frac{3}{2}\|\nabla d\|_{L^3}^\frac{5}{2}\|\nabla u\|_{L^6}
        \leq C\|\nabla d\|_{L^3}^3\|\nabla^3d\|_{L^2}\|\nabla u\|_{H^1}\\
        &\leq \frac{1}{16}\|\sqrt{\rho}u_t\|_{L^2}^2+C\bar\rho^{-\alpha}\|\nabla d\|_{L^3}^2\|\nabla^3d\|_{L^2}^2+C\bar\rho^{4-3\alpha}\|\nabla u\|_{L^2}^6,
    \end{aligned}
\end{equation}
where we have used \eqref{GN-d-2}.

Then \eqref{a1} together with \eqref{51}-\eqref{I-5} implies that 
\begin{equation}\label{t1}
\begin{aligned}
&\frac{d}{dt}\left(\int \mu\rho^\alpha |\mathcal{D}(u)|^2dx - \lambda \int\nabla d\odot\nabla d:\nabla u dx \right)+  \frac{3}{8}\int\rho |u_t|^2dx\\
&\leq C\bar\rho^{1-2\alpha}\|\nabla d\|_{L^3}^2\|\sqrt{\rho} u_t\|_{L^2}^2+C(\bar\rho^{3-2\alpha}+\bar \rho^{\frac{6}{q}(1-\alpha)}+\bar\rho^{-\frac{q+6}{3q-6}+\frac{12-2q}{3q-6}(1-\alpha)})\|\nabla u\|_{L^2}^4\\
&\quad +C\|\nabla d\|_{L^3}^2\|\nabla^3d\|_{L^2}^2\\
&\leq c_5 \bar\rho^{1-2\alpha} \|\sqrt{\rho} u_t\|_{L^2}^2
+C\bar\rho^{A_0}\|\nabla u\|_{L^2}^4
+C \|\nabla^3d\|_{L^2}^2\\
&\leq C\bar\rho^{A_0}\|\nabla u\|_{L^2}^4
+C \|\nabla^3d\|_{L^2}^2,
%&\leq C\|\nabla d\|_{L^3}^2\|\sqrt{\rho} u_t\|_{L^2}^2
%+C\bar\rho^{A_0}\|\nabla u\|_{L^2}^2
%+C\|\nabla d\|_{L^3}^2\|\nabla^3d\|_{L^2}^2,
\end{aligned}
\end{equation}
since $\frac{2(5q-6)}{3(q-2)}\geq 4$ and let
$$\bar\rho> (8c_5)^{\frac{1}{2\alpha-1}},$$ 
and
\begin{equation}
    A_0=\max\left\{3-2\alpha, {-\frac{q+6}{3q-6}+\frac{12-2q}{3q-6}(1-\alpha)} \right\}<2\alpha-1.
\end{equation}

Note that 
\begin{align*}
    \frac{\mu}{2} \bar \rho^{\alpha} \|\nabla u\|_{L^2}^2 - C \bar \rho^{-\alpha} \|\nabla d\|_{L^4}^4
    \le &\int \mu\rho^\alpha |\mathcal{D}(u)|^2dx -\lambda \int\nabla d\odot\nabla d:\nabla u dx \\
    \le &C \bar \rho^{\alpha} \|\nabla u\|_{L^2}^2 + C \bar \rho^{-\alpha} \|\nabla d\|_{L^4}^4\\
    \le &C \bar \rho^{\alpha} \|\nabla u\|_{L^2}^2 + C \bar \rho^{-\alpha} \|\nabla d\|_{L^2} \|\nabla^2 d\|_{L^2}^3.
\end{align*}
Integrating \eqref{t1} with respect to $t$ over $(0, T]$, we can get from \eqref{deg}, \eqref{2.3}, \eqref{basic-est} and \eqref{d2d} that
\begin{equation}
\begin{aligned}
&\sup_{t\in[0,T] } \bar\rho^{\alpha}\|\nabla u \|_{L^2
}^2 +  \int_0^{T}  \|\sqrt{\rho} u_t \|_{L^2}^2dt\\
 &\leq \bar\rho^{\alpha} \|\nabla u_0\|_{L^2}^2+ C\bar\rho^{-\alpha} \sup_{t\in[0,T] }\|\nabla d\|_{L^2} \|\nabla^2 d\|_{L^2}^3\\
 &\quad +C\bar{\rho}^{A_0}\int_0^T\|\nabla u\|_{L^2}^2dt + C\int_{0}^{T} \|\nabla^3d\|_{L^2}^2 dt
\\
 &\leq M\bar\rho^{\alpha}+c_6\bar{\rho}^{A_0+1-\alpha}+c_7\\
 &\leq 2M\bar\rho^{\alpha}  =2\mathcal{E}_u(0),
\end{aligned}
\end{equation}
provided 
\begin{equation}
    M \triangleq \|\nabla u_0\|_{L^2}^2 ,
\end{equation}
and
\begin{equation}
    \bar{\rho}\geq \max\biggl\{(8c_5)^{\frac{1}{2\alpha-1}},\biggl(\frac{c_6}{M}\biggr)^\frac{1}{2\alpha-A_0-1},\biggl(\frac{c_7}{M}\biggr)^\frac{1}{\alpha}\biggr\}\triangleq \Lambda_1.
\end{equation}

Finally, multiplying \eqref{t1} by $t$ and applying Gronwall’s inequality and using \eqref{basic-est}, \eqref{d2d}, and \eqref{tdd}, we obtain
\begin{equation}
\begin{aligned}
&\sup_{t\in[0,T]}t \bar\rho^{\alpha}\|\nabla u \|_{L^2}^2
+\int_0^{T} t\|\sqrt{\rho} u_t \|_{L^2}^2 dt\\
&\leq C \Bigl( \sup_{t\in[0,T]}t \bar\rho^{-\alpha}\|\nabla d\|_{L^2} \|\nabla^2 d\|_{L^2}^3 +\bar \rho^{A_0-2\alpha} \int_{0}^{T}   t \|\nabla d\|_{L^2} \|\nabla^2 d\|_{L^2}^3\|\nabla u \|_{L^2}^2 dt \\
&\quad + \int_{0}^{T}  (\bar \rho^\alpha \|\nabla u \|_{L^2}^2 + \bar \rho^{-\alpha} \|\nabla d\|_{L^2} \|\nabla^2 d\|_{L^2}^3 ) dt  +  \int_{0}^{T} t \|\nabla^3d\|_{L^2}^2 dt \Bigl)
\\
& \quad \cdot\exp{\bigg\{C\bar\rho^{A_0-\alpha}
\int_0^{T}\|\nabla u\|_{L^2}^2dt\bigg\}}\\
&\leq C \Bigl( \sup_{t\in[0,T]} \bar\rho^{-\alpha} t \|\nabla^2 d\|_{L^2}^2 +\bar \rho^{A_0-2\alpha} \sup_{t\in[0,T]} t \|\nabla^2 d\|_{L^2}^2 \int_{0}^{T}   \|\nabla u \|_{L^2}^2 dt \\
&\quad + \int_{0}^{T}  (\bar \rho^\alpha \|\nabla u \|_{L^2}^2 + \bar \rho^{-\alpha} \|\nabla^2 d\|_{L^2}^3 ) dt  +  \int_{0}^{T} t \|\nabla^3d\|_{L^2}^2 dt \Bigl)
\\
& \quad \cdot\exp{\{C\bar\rho^{A_0+1-2\alpha}\}}\\
&\leq C(\bar\rho+1) \cdot\exp{\{C\bar\rho^{A_0+1-2\alpha}\}}\\
&\leq C\bar\rho,
\end{aligned}
\end{equation}
due to $\alpha>1$.
\end{proof}

In order to close the estimate of $\mathcal{E}_\rho$, we need the following time-weight estimates.
\begin{lemma}
    It holds that 
\begin{equation}\label{trut}
    \sup_{0\le t\le T} t \left(\|\sqrt{\rho} u_t \|_{L^2}^2
    +\|\nabla d_t\|_{L^2}^2\right) 
    +\int_{0}^{T} t\left(\bar\rho^\alpha\|\nabla u_t\|^2_{L^2}
    +\|\nabla^2 d_t\|_{L^2}^2\right) dt\\
    \leq C\bar\rho^\alpha,
\end{equation}
and
\begin{equation}\label{ttrut}
    \sup_{0\le t\le T} t^2 \left(\|\sqrt{\rho} u_t \|_{L^2}^2
    +\|\nabla d_t\|_{L^2}^2\right) 
    +\int_{0}^{T} t^2\left(\bar\rho^\alpha\|\nabla u_t\|^2_{L^2}
    +\|\nabla^2 d_t\|_{L^2}^2\right) dt
    \leq C\bar\rho.
\end{equation}
\end{lemma}

\begin{proof}
Take the $t$-derivative of the momentum equations, multiply the resulting equation by $tu_t$, and after integrating by parts, we have that 
\begin{equation}\label{k2}
    \begin{aligned}
        &\frac{t}{2} \frac{d}{dt} \int \rho |u_t|^2 dx + 2\mu t \int\rho^\alpha |\mathcal{D}(u)_t|^2dx\\
        =&-t\int  \rho u_t \cdot \nabla u \cdot u_t dx 
        +t\int  \dv(\rho u) |u_t|^2 dx 
        +t\int  \dv(\rho u) u \cdot \nabla u \cdot u_t dx  \\
        &+2\mu t \int (\rho^\alpha)_t \mathcal{D}(u): \mathcal{D}(u)_tdx
        +2t \nu \int \mathcal{D}_{ij}(u_t)\partial_id\cdot\partial_jd_tdx\\
        =&\sum_{j=1}^{5} J_i.
    \end{aligned}
\end{equation}

It follows from H\"older and Sobolev inequalities that
\begin{equation}
    \begin{aligned}
        J_1\leq& C t \int |\rho \nabla u|  |u_t|^2 dx  \\
        \le &C t\bar\rho\|\nabla u\|_{L^2}\|u_t\|_{L^4}^2 \\
        \le &Ct  \bar\rho^{\frac34} \|\nabla u\|_{L^2} \|\sqrt{\rho} u_t \|_{L^2}^{\frac 12}  \| \nabla u_t \|_{L^2}^{\frac 32}\\
        \le & \frac{\mu}{8} t \bar\rho^{\alpha} \|\nabla u_t\|_{L^2}^2 
        +Ct\bar \rho^{3-3\alpha}\|\sqrt{\rho} u_t \|_{L^2}^2 \|\nabla u\|_{L^2}^4\\
        \le & \frac{\mu}{8} t \bar\rho^{\alpha} \|\nabla u_t\|_{L^2}^2 
        +Ct \|\sqrt{\rho} u_t \|_{L^2}^2 \|\nabla u\|_{L^2}^2,
    \end{aligned}
\end{equation}
% \begin{equation}
%     \begin{aligned}
%         J_1\leq& C t \int |\rho \nabla u|  |u_t|^2 dx  \\
%         \le &C t\bar \rho^{\frac12}\|\sqrt{\rho} u_t \|_{L^2} \|\nabla u\|_{L^3}\| u_t \|_{L^6} \\
%         \le &Ct  \bar\rho^{\frac12}\|\sqrt{\rho} u_t \|_{L^2} \|\nabla u\|_{L^2}^{\frac12} \|\nabla u\|_{L^6}^{\frac12} \| \nabla u_t \|_{L^2}\\
%         \le & \frac{\mu}{8} t \bar\rho^{\alpha} \|\nabla u_t\|_{L^2}^2 
%         +Ct\bar \rho^{1-\alpha}\|\sqrt{\rho} u_t \|_{L^2}^2 \|\nabla u\|_{L^2} \|\nabla u\|_{L^6}\\
%         \leq & \frac{\mu}{8} t \bar\rho^{\alpha} \|\nabla u_t\|_{L^2}^2 
%         +Ct\bar\rho^{-\alpha} \|\nabla u\|_{L^2} \|\sqrt{\rho} u_t \|_{L^2}^2\|\nabla u\|_{H^1}\\
%         \leq & \frac{\mu}{8} t \bar\rho^{\alpha} \|\nabla u_t\|_{L^2}^2
%         +Ct\bar\rho^{-\alpha} \|\nabla u\|_{L^2} \|\sqrt{\rho} u_t \|_{L^2}^2\\
%         &\cdot(\bar\rho^{\frac{1}{2}-\alpha}\|\sqrt{\rho} u_t\|_{L^2}
%         +\bar\rho^{2-2\alpha}\|\nabla u\|_{L^2}^3\color{red}{+\bar\rho^{-\alpha}\|\nabla d\|_{L^3}\|\nabla^3d\|_{L^2}})\\
%         \leq & \frac{\mu}{8} t \bar\rho^{\alpha} \|\nabla u_t\|_{L^2}^2
%         +Ct\bar\rho^{\frac{1}{2}-2\alpha} \|\nabla u\|_{L^2} \|\sqrt{\rho} u_t \|_{L^2}^3
%         +Ct\bar\rho^{2-3\alpha} \|\nabla u\|_{L^2}^4 \|\sqrt{\rho} u_t \|_{L^2}^2\\
%         &\color{red}{\quad+Ct\bar\rho^{-2\alpha}\|\nabla u\|_{L^2}^2\|\sqrt{\rho}u_t\|_{L^2}^4
%         +Ct\bar\rho^{-2\alpha}\|\nabla d\|_{L^3}^2\|\nabla^3d\|_{L^2}^2},
%     \end{aligned}
% \end{equation}
also $J_2$ can be estimated as follows
\begin{equation}
    \begin{aligned}
        J_2\leq& Ct\int|\nabla\rho\cdot u||u_t|^2dx\\
        \leq &C t\|\nabla \rho\|_{L^q} \|u \|_{L^{\frac{2q}{q-2}}} \| u_t \|_{L^4}^{2}\\
        \leq &Ct \bar \rho^{-\frac14} \|\nabla \rho\|_{L^q} \|u \|_{L^{2}}^{\frac{q-3}{q}} \|\nabla u \|_{L^{2}}^{\frac{3}{q}} \|\sqrt{\rho} u_t \|_{L^2}^{\frac12} \|\nabla u_t \|_{L^2}^{\frac32}  \\
        \le & \frac{\mu}{8} t \bar\rho^{\alpha} \|\nabla u_t\|_{L^2}^2 
        +C t \bar \rho^{-3\alpha-1} \|\nabla \rho\|_{L^q}^4 \|\nabla u \|_{L^{2}}^{\frac{12}{q}} \|\sqrt{\rho} u_t \|_{L^2}^2 \\
        \leq & \frac{\mu}{8} t \bar\rho^{\alpha} \|\nabla u_t\|_{L^2}^2 
        + Ct \mathcal{E}_\rho(0)^4 \bar\rho^{-3\alpha-1} \|\nabla u \|_{L^2}^\frac{12}{q} \|\sqrt{\rho} u_t \|_{L^2}^2\\
        \leq & \frac{\mu}{8} t \bar\rho^{\alpha} \|\nabla u_t\|_{L^2}^2 
        + Ct \|\nabla u \|_{L^2}^2 \|\sqrt{\rho} u_t \|_{L^2}^2.
    \end{aligned}
\end{equation}
Similarly, $J_3$ can be bounded by
\begin{equation}
    \begin{aligned}
        J_3\leq & t\int |\nabla \rho \cdot u| |u| | \nabla u| |u_t| dx\\
        \leq & Ct \|\nabla \rho\|_{L^q}  \|u\|_{L^6}^2  \|\nabla u\|_{L^{\frac{2q}{q-2}}} \|u_t\|_{L^6} \\
        \leq & \frac{\mu}{8} t \bar\rho^{\alpha} \|\nabla u_t\|_{L^2}^2 + Ct \bar\rho^{ - \alpha} \|\nabla \rho \|_{L^q}^2 \|\nabla u\|_{L^2}^\frac{6(q-1)}{q}\|\nabla u\|_{H^1}^\frac{6}{q}\\
        \leq & \frac{\mu}{8} t \bar\rho^{\alpha} \|\nabla u_t\|_{L^2}^2 
        +C\mathcal{E}_\rho(0)^2t \bar\rho^{ - \alpha} \|\nabla u\|_{L^2}^\frac{6(q-1)}{q}\\
        &\cdot(\bar\rho^{\frac{1}{2}-\alpha}\|\sqrt{\rho} u_t\|_{L^2}
        +\bar\rho^{2-2\alpha}\|\nabla u\|_{L^2}^3+\bar\rho^{-\alpha}\|\nabla d\|_{L^3}\|\nabla^3d\|_{L^2})^\frac{6}{q}\\
        \leq & \frac{\mu}{8} t \bar\rho^{\alpha} \|\nabla u_t\|_{L^2}^2
        +Ct\rho^{ -\alpha+2(1-\alpha)\frac{6}{q}}
        \|\nabla u\|_{L^2}^{6+\frac{12}{q}}\\
        &+Ct\bar\rho^{ -\alpha+\frac{6}{q}(\frac{1}{2}-\alpha)}\left(\|\nabla u\|_{L^2}^2
        \|\sqrt{\rho} u_t\|_{L^2}^2+\|\nabla u\|_{L^2}^{6+\frac{6}{q-3}}\right)\\
        &+Ct\bar\rho^{-\alpha(1+\frac{6}{q})}\|\nabla u\|_{L^2}^{6+\frac{6}{q-3}}
        +Ct\bar\rho^{-\alpha(1+\frac{6}{q})} \|\nabla u\|_{L^2}^{2}\|\nabla d\|_{L^3}^2\|\nabla ^3d\|_{L^2}^2\\
        \leq & \frac{\mu}{8} t \bar\rho^{\alpha} \|\nabla u_t\|_{L^2}^2
        +Ct  \bar \rho^{ -\alpha} \|\nabla u\|_{L^2}^{8} +Ct \|\nabla u\|_{L^2}^2
        \|\sqrt{\rho} u_t\|_{L^2}^2 \\
        &  +Ct\bar\rho^{-\alpha(1+\frac{6}{q})} \|\nabla u\|_{L^2}^{2}\|\nabla d\|_{L^3}^2\|\nabla ^3d\|_{L^2}^2.
    \end{aligned}
\end{equation}
Using \eqref{rhoat}, we have
\begin{equation}\label{J4}
    \begin{aligned}
        J_4\leq & Ct\bar\rho^{\alpha-1}\int |\nabla \rho \cdot u| |\nabla  u||\nabla  u_t| dx\\
        \leq & Ct\bar \rho^{\alpha-1}  \|\nabla \rho\|_{L^q}  \|u\|_{L^\frac{3q}{q-3}}  \|\nabla u\|_{L^6} \|\nabla u_t\|_{L^2}\\
        \leq & \frac{\mu}{8} t\bar \rho^{\alpha} \|\nabla u_t\|_{L^2}^2
        +Ct\bar \rho^{\alpha-2} \|\nabla \rho \|_{L^q}^2  \|\nabla u\|_{L^2}^{3-\frac{6}{q}} \|\nabla u\|_{L^6}^{1+\frac{6}{q}}\\
        \leq & \frac{\mu}{8}t \bar\rho^{\alpha} \|\nabla u_t\|_{L^2}^2 
        + C\mathcal{E}_\rho(0)^2 t\bar \rho^{\alpha-2}  \|\nabla u\|_{L^2}^{3-\frac{6}{q}} \|\nabla u\|_{H^1}^{1+\frac{6}{q}}\\
        \leq & \frac{\mu}{8}t \bar \rho^{\alpha} \|\nabla u_t\|_{L^2}^2 
        +Ct \bar\rho^{\alpha-2}  \|\nabla u\|_{L^2}^{3-\frac{6}{q}}\\
        &\cdot(\bar\rho^{\frac{1}{2}-\alpha}\|\sqrt{\rho} u_t\|_{L^2}
        +\bar\rho^{2-2\alpha}\|\nabla u\|_{L^2}^3+\bar\rho^{-\alpha}\|\nabla d\|_{L^3}\|\nabla^3d\|_{L^2})^{1+\frac{6}{q}}\\
        \leq & \frac{\mu}{8}t \bar \rho^{\alpha} \|\nabla u_t\|_{L^2}^2  
        +Ct \bar\rho^{-\frac{3}{2}+(\frac{1}{2}-\alpha)\frac{6}{q}}\|\nabla u\|_{L^2}^{3-\frac{6}{q}}\|\sqrt{\rho} u_t\|_{L^2}^{1+\frac{6}{q}}
        +Ct \bar\rho^{-\alpha+2(1-\alpha)\frac{6}{q}}\|\nabla u\|_{L^2}^{6+\frac{12}{q}}\\
        %&\quad+\color{red}{Ct\bar\rho^{-2-\frac{6\alpha}{q}}(\|\nabla d\|_{L^3}^2\|\nabla u\|_{L^2}^2\|\nabla^3d\|_{L^2}^2+\|\nabla d\|_{L^3}^4\|\nabla^3d\|_{L^2}^4)}\\
        &+Ct\bar\rho^{-2-\frac{6\alpha}{q}}\|\nabla u\|_{L^2}^{3-\frac{6}{q}} \|\nabla d\|_{L^3}^{1+\frac{6}{q}}\|\nabla^3d\|_{L^2}^{1+\frac{6}{q}}\\
        \leq & \frac{\mu}{8}t \bar \rho^{\alpha} \|\nabla u_t\|_{L^2}^2  
        +Ct \bar \rho^{ -\alpha} \|\nabla u\|_{L^2}^{8}\\
        &+Ct \bar\rho^{-\frac{3}{2}+(\frac{1}{2}-\alpha)\frac{6}{q}}\|\nabla u\|_{L^2}^{3-\frac{6}{q}}\|\sqrt{\rho} u_t\|_{L^2}^{1+\frac{6}{q}}
        +Ct\bar\rho^{-2-\frac{6\alpha}{q}}\|\nabla u\|_{L^2}^{3-\frac{6}{q}} \|\nabla^3d\|_{L^2}^{1+\frac{6}{q}}.
    \end{aligned}
\end{equation}
Next,
\begin{align}\label{J5}
    J_5&\leq Ct\int |\nabla u_t| |\nabla d||\nabla  d_t| dx
    \leq Ct\|\nabla u_t\|_{L^2}\|\nabla d\|_{L^3}\|\nabla d_t\|_{L^6}\nonumber\\
    &\leq \frac{1}{16}t\bar \rho^{\alpha} \|\nabla u_t\|_{L^2}^2  
    +Ct\bar\rho^{-\alpha}\|\nabla d\|_{L^3}^2\| \nabla^2 d_t\|_{L^2}^2.
\end{align}

Operating $\partial_t$ to $\eqref{ins}_3$, and multiplying it by $-\Delta d_t$, by integrating the resulting equality
by parts, we get
\begin{equation}\label{ddt}
    \begin{aligned}
    &\frac{1}{2}\frac{d}{dt}\|\nabla d_t\|_{L^2}^2
    +\lambda \|\Delta d_t\|_{L^2}^2\\
    &=\int u_t\cdot\nabla d\cdot \Delta d_t dx
    -\int \nabla u:\nabla d_t\odot\nabla d_t dx-\lambda \int(|\nabla d|^2d)_t\cdot \Delta d_t dx\\
    &\leq C (\|u_t\|_{L^6}\|\nabla d\|_{L^3}+\|\nabla d_t\|_{L^6}\|\nabla d\|_{L^3}+\|\nabla d\|_{L^6}^2\|d_t\|_{L^6})\|\Delta d_t\|_{L^2}+C\|\nabla u\|_{L^2}\|\nabla d_t\|_{L^4}^2\\
    &\leq C (\|\nabla u_t\|_{L^2}\|\nabla d\|_{L^3}+\|\nabla^2 d_t\|_{L^2}\|\nabla d\|_{L^3}+\|\nabla^2 d\|_{L^2}^2\|d_t\|_{H^1})\|\Delta d_t\|_{L^2}\\
    &\quad +C\|\nabla u\|_{L^2}\|\nabla d_t\|_{L^2}^{\frac{1}{2}}\|\nabla^2d_t\|_{L^2}^\frac{3}{2}\\
    &\leq \eta\|\nabla^2 d_t\|_{L^2}^2
    +\frac{1}{4}\| \Delta d_{t}\|_{L^2}^2+C\|\nabla d\|_{L^3}^2\|\nabla u_t\|_{L^2}^2
    +c_8\|\nabla d\|_{L^3}^2\|\Delta d_t\|_{L^2}^2\\
    &\quad+C(\|\nabla u\|_{L^2}^4+\|\nabla^2 d\|_{L^2}^4)\|\nabla d_t\|_{L^2}^2\\
    &\quad +C\|\nabla^2 d\|_{L^2}^6+C\|\nabla d\|_{L^3}^2(\|\nabla u\|_{L^2}^2+\|\nabla^2d\|_{L^2}^2)\|\nabla^2 d\|_{L^2}^4,
\end{aligned}
\end{equation}
where we have used that 
\begin{align*}
    \int u\cdot\nabla d_t\cdot\Delta d_t dx=-\int\partial_j u_i\partial_id_t\cdot\partial_j d_t dx-\int (u\cdot \nabla) \partial_j d_t\cdot \partial_j d_t dx=-\int \nabla u:\nabla d_t\odot\nabla d_t dx,
\end{align*}
and
\begin{align*}
    \|d_t\|_{L^2}^2&\leq C(\|\Delta d\|_{L^2}^2+\|u\cdot\nabla d\|_{L^2}^2+\|\nabla d\|_{L^4}^4)\\
    &\leq C\|\nabla ^2 d\|_{L^2}^2+C\|\nabla d\|_{L^3}^2(\|\nabla u\|_{L^2}^2+\|\nabla^2d\|_{L^2}^2).
\end{align*}
By elliptic estimates $\|\nabla d_t\|_{L^2}^2\leq C_3^2\|\Delta d_t\|_{L^2}^2$, which follows from Lemma \ref{Nuemann_elliptic}, and choosing $\eta$ small enough and 
$$\delta<\frac{1}{4c_8},$$ 
we have
\begin{equation}\label{d_t-2nd}
    \begin{aligned}
      \frac{d}{dt}\|\nabla d_t\|_{L^2}^2
    +\|\nabla^2 d_{t}\|_{L^2}^2\leq &C\|\nabla d\|_{L^3}^2\|\nabla u_t\|_{L^2}^2+C(\|\nabla u\|_{L^2}^4+\|\nabla^2 d\|_{L^2}^4)\|\nabla d_t\|_{L^2}^2\\
    &+C\|\nabla^2 d\|_{L^2}^6+C\|\nabla d\|_{L^3}^2(\|\nabla u\|_{L^2}^2+\|\nabla^2d\|_{L^2}^2)\|\nabla^2 d\|_{L^2}^4\\
    \leq &C\|\nabla d\|_{L^3}^2\|\nabla u_t\|_{L^2}^2+C(\|\nabla u\|_{L^2}^4+\|\nabla^2 d\|_{L^2}^4)\|\nabla d_t\|_{L^2}^2\\
    &+C(\|\nabla u\|_{L^2}^2+\|\nabla^2d\|_{L^2}^2)\|\nabla^2 d\|_{L^2}^4.   
    \end{aligned}
\end{equation}

Combining all the above estimates \eqref{k2}–\eqref{J5}, \eqref{d_t-2nd}
\eqref{a1} and \eqref{basic-est}, we deduce
\begin{equation}\label{k3}
    \begin{aligned}
        &\frac{d}{dt}t \int \left(\rho |u_t|^2+ \vert\nabla d_t\vert^2\right) dx 
        + t \int\left(\bar\rho^\alpha|\nabla u_t|^2 + |\nabla^2 d_t|^2\right)dx\\
        % \leq & \int \left(\rho |u_t|^2 +\vert\nabla d_t\vert^2\right) dx \\
        % & +Ct\bar\rho^{3-3\alpha} \|\nabla u\|_{L^2}^4 \|\sqrt{\rho} u_t \|_{L^2}^2 + Ct  \bar\rho^{-3\alpha-1} \|\nabla u \|_{L^2}^2 \|\sqrt{\rho} u_t \|_{L^2}^2\\
        % &+Ct\bar\rho^{ -\alpha+\frac{6}{q}(\frac{1}{2}-\alpha)}\|\nabla u\|_{L^2}^2
        % \|\sqrt{\rho} u_t\|_{L^2}^2\\
        % &+Ct\rho^{ -\alpha} \|\nabla u\|_{L^2}^{8}
        % +Ct \bar\rho^{-\frac{3}{2}+(\frac{1}{2}-\alpha)\frac{6}{q}}\|\nabla u\|_{L^2}^{3-\frac{6}{q}}\|\sqrt{\rho} u_t\|_{L^2}^{1+\frac{6}{q}}\\
        % & + {\color{red} Ct\bar\rho^{-\alpha(1+\frac{6}{q})} \|\nabla u\|_{L^2}^{2}\|\nabla d\|_{L^3}^2\|\nabla ^3d\|_{L^2}^2}
        % +\color{red}{Ct\bar\rho^{-2-\frac{6\alpha}{q}}\|\nabla u\|_{L^2}^{3-\frac{6}{q}} \|\nabla^3d\|_{L^2}^{1+\frac{6}{q}} }\\
        % &+Ct\bar\rho^{-\alpha}\|\nabla d\|_{L^3}^2\| \nabla^2 d_t\|_{L^2}^2\\
        % &+ Ct\|\nabla d\|_{L^3}^2\|\nabla u_t\|_{L^2}^2+Ct(\|\nabla u\|_{L^2}^4+\|\nabla^2 d\|_{L^2}^4)\|\nabla d_t\|_{L^2}^2\\
        % &+ Ct(\|\nabla u\|_{L^2}^2+\|\nabla^2d\|_{L^2}^2)\|\nabla^2 d\|_{L^2}^4\\
        \leq &\int \left(\rho |u_t|^2+\vert\nabla d_t\vert^2\right) dx 
        + Ct \|\nabla u\|_{L^2}^2 \|\sqrt{\rho} u_t \|_{L^2}^2 + Ct \bar \rho^{ -\alpha} \|\nabla u\|_{L^2}^8\\
        &+ Ct \bar\rho^{-\frac{3}{2}+(\frac{1}{2}-\alpha)\frac{6}{q}}\|\nabla u\|_{L^2}^{3-\frac{6}{q}}\|\sqrt{\rho} u_t\|_{L^2}^{1+\frac{6}{q}}\\
        & + Ct\bar\rho^{-\alpha(1+\frac{6}{q})} \|\nabla u\|_{L^2}^{2}\|\nabla d\|_{L^3}^2\|\nabla ^3d\|_{L^2}^2
        +Ct\bar\rho^{-2-\frac{6\alpha}{q}}\|\nabla u\|_{L^2}^{3-\frac{6}{q}} \|\nabla^3d\|_{L^2}^{1+\frac{6}{q}}\\
        &+Ct\bar\rho^{-\alpha} \left( \|\nabla^2 d_t\|_{L^2}^2 +  \bar \rho^{\alpha} \|\nabla u_t\|_{L^2}^2\right)\\
        & +Ct(\|\nabla u\|_{L^2}^4+\|\nabla^2 d\|_{L^2}^4)\|\nabla d_t\|_{L^2}^2 + Ct(\|\nabla u\|_{L^2}^2+\|\nabla^2d\|_{L^2}^2)\|\nabla^2 d\|_{L^2}^4,
    \end{aligned}
\end{equation}
due to $q\in (3,6)$, \eqref{a1}. 

Since $w=\nabla d$ satisfies the following elliptic system
\begin{equation}\label{dd_ellip}
	\begin{cases}
		-\lambda \Delta w=\nabla d_t-\nabla(u\cdot\nabla d)+ \lambda \nabla(\vert\nabla d\vert^2 d),\\
		w\cdot n=0,~\mathrm{curl}w\times n=0,
	\end{cases} 
\end{equation}
By using elliptic estimates Lemma \ref{slip_estimate} for $\eqref{dd_ellip}$, we have
\begin{equation*}
    \begin{aligned}
        \|\nabla^3d\|_{L^2}^2
        &\leq C\left(\|\nabla d_t\|_{L^2}^2+\|\nabla u\nabla d\|_{L^2}^2+\|u\cdot\nabla \nabla d\|_{L^2}^2+\|\nabla d\nabla^2 d\|_{L^2}^2+\|\nabla d\|_{L^6}^6\right)\\
        &\leq C\|\nabla d_t\|_{L^2}^2+C(\|\nabla u\|_{L^2}^2\|\nabla d\|_{L^\infty}^2+\|u\|_{L^6}^2\|\nabla ^2d\|_{L^3}^2)\\
        &\quad+C(\|\nabla d\|_{L^3}^2\|\nabla^2d\|_{L^6}^2+\|\nabla d\|_{L^\infty}^3\|\nabla d\|_{L^3}^3)\\
        &\leq C\|\nabla d_t\|_{L^2}^2+C\|\nabla u\|_{L^2}^2\|\nabla d\|_{L^3}^\frac{2}{3}\|\nabla^3d\|_{L^2}^\frac{4}{3}+C(\|\nabla d\|_{L^3}^2+\|\nabla d\|_{L^3}^4)\|\nabla^3d\|_{L^2}^2\\
        &\leq \frac{1}{4}\|\nabla^3d\|_{L^2}^2+ c_9 (\|\nabla d\|_{L^3}^2+\|\nabla d\|_{L^3}^4)\|\nabla^3d\|_{L^2}^2+C\|\nabla d_t\|_{L^2}^2\\
        &\quad+C\|\nabla d\|_{L^3}^2\|\nabla u\|_{L^2}^6.
    \end{aligned}
\end{equation*}
Therefore, by choosing 
$$\delta<\frac{1}{4c_9},$$
we have
\begin{equation}\label{d-3rd}
    \begin{aligned}
      \|\nabla^3d\|_{L^2}^2&\leq C\|\nabla d_t\|_{L^2}^2+C\|\nabla d\|_{L^3}^2\|\nabla u\|_{L^2}^6.  
    \end{aligned}
\end{equation}

Plugging \eqref{d-3rd} into \eqref{k3}, we have
\begin{equation}
    \begin{aligned}
        &\frac{d}{dt}t \int \left(\rho |u_t|^2+ \vert\nabla d_t\vert^2\right) dx 
        + t \int\left(\bar\rho^\alpha|\nabla u_t|^2 + |\nabla^2 d_t|^2\right)dx\\
        \leq &\int \left(\rho |u_t|^2+\vert\nabla d_t\vert^2\right) dx 
        + Ct \|\nabla u\|_{L^2}^2 \|\sqrt{\rho} u_t \|_{L^2}^2 + Ct \bar \rho^{ -\alpha} \|\nabla u\|_{L^2}^8\\
        &+ Ct \bar\rho^{-\frac{3}{2}+(\frac{1}{2}-\alpha)\frac{6}{q}}\|\nabla u\|_{L^2}^{3-\frac{6}{q}}\|\sqrt{\rho} u_t\|_{L^2}^{1+\frac{6}{q}}\\
        & + Ct\bar\rho^{-\alpha(1+\frac{6}{q})} \|\nabla u\|_{L^2}^{2}\|\nabla d\|_{L^3}^2\|\nabla ^3d\|_{L^2}^2\\
        &+Ct\bar\rho^{-2-\frac{6\alpha}{q}}\|\nabla u\|_{L^2}^{3-\frac{6}{q}} \|\nabla^3d\|_{L^2}^{\frac{6}{q}-1}
        \left(\|\nabla d_t\|_{L^2}^2+\|\nabla d\|_{L^3}^2\|\nabla u\|_{L^2}^6\right)\\
        &+Ct\bar\rho^{-\alpha} \left( \| \nabla^2 d_t\|_{L^2}^2 + \bar \rho^{\alpha} \|\nabla u_t\|_{L^2}^2\right)\\
        & +Ct(\|\nabla u\|_{L^2}^4+\|\nabla^2 d\|_{L^2}^4)\|\nabla d_t\|_{L^2}^2 + Ct(\|\nabla u\|_{L^2}^2+\|\nabla^2d\|_{L^2}^2)\|\nabla^2 d\|_{L^2}^4\\
        \leq &\int \left(\rho |u_t|^2+\vert\nabla d_t\vert^2\right) dx 
        + Ct \|\nabla u\|_{L^2}^2 \|\sqrt{\rho} u_t \|_{L^2}^2 + Ct\bar \rho^{ -\alpha} \|\nabla u\|_{L^2}^8\\
        &+ Ct \bar\rho^{-\frac{3}{2}+(\frac{1}{2}-\alpha)\frac{6}{q}}\|\nabla u\|_{L^2}^{3-\frac{6}{q}}\|\sqrt{\rho} u_t\|_{L^2}^{1+\frac{6}{q}} + c_{10} t\bar\rho^{-\alpha} \left( \|\nabla^2 d_t\|_{L^2}^2 +  \bar \rho^{\alpha} \|\nabla u_t\|_{L^2}^2\right)\\
        & +  Ct\bar\rho^{-\alpha(1+\frac{6}{q})} \|\nabla u\|_{L^2}^{2} \|\nabla d\|_{L^3}^2\|\nabla ^3d\|_{L^2}^2\\
        %+Ct\bar\rho^{-2-\frac{6\alpha}{q}}\|\nabla u\|_{L^2}^{9-\frac{6}{q}} \|\nabla^3d\|_{L^2}^{\frac{6}{q}-1} \\
        & +Ct(\|\nabla u\|_{L^2}^4+\|\nabla^2 d\|_{L^2}^4 + \|\nabla u\|_{L^2}^{2} + \|\nabla^3d\|_{L^2}^{2} )\|\nabla d_t\|_{L^2}^2 \\
        & + Ct(\|\nabla u\|_{L^2}^2+\|\nabla^2d\|_{L^2}^2)\|\nabla^2 d\|_{L^2}^4,
    \end{aligned}
\end{equation}
where we have used 
\begin{align*}
   \bar\rho^{-2-\frac{6\alpha}{q}}\|\nabla u\|_{L^2}^{9-\frac{6}{q}} \|\nabla^3d\|_{L^2}^{\frac{6}{q}-1} &\le C\bar\rho^{-\alpha(1+\frac{6}{q})} \|\nabla u\|_{L^2}^{2} \|\nabla d\|_{L^3}^2\|\nabla ^3d\|_{L^2}^2 + C \bar \rho^{-\alpha}\|\nabla u\|_{L^2}^{\frac{4(5q-6)}{3(q-2)}},
\end{align*}
due to $\frac{4(5q-6)}{3(q-2)}\ge 8$ for $q\in(3,6)$ and $\bar \rho>1$.
Therefore, choosing 
$$\bar \rho > \frac{1}{(4c_{10})^{\alpha}},$$
Gronwall's inequality yields
\begin{equation}
    \begin{aligned}
        &\sup_{0\le t\le T} t \left(\|\sqrt{\rho} u_t \|_{L^2}^2
        + \|\nabla d_t\|_{L^2}^2\right) 
        +\int_{0}^{T} t\left(\bar\rho^\alpha\|\nabla u_t\|^2_{L^2}
        + \|\Delta d_t\|_{L^2}^2\right) dt\\
        \leq & \bigg\{C\int_0^T t \bar\rho^{ -\alpha } 
        \|\nabla u\|_{L^2}^8  dt 
        +C\int_0^T \left(\|\sqrt{\rho} u_t\|_{L^2}^2
        + \|\nabla d_t\|_{L^2}^2 \right)dt\\
        &+C\bar\rho^{-\alpha}\int_0^T t\bar\rho^{-\alpha(1+\frac{6}{q})} \|\nabla u\|_{L^2}^{2}\|\nabla d\|_{L^3}^2\|\nabla ^3d\|_{L^2}^2 dt \\ 
        %+\color{red}{t\bar\rho^{-2-\frac{6\alpha}{q}}\|\nabla u\|_{L^2}^{9-\frac{6}{q}} \|\nabla^3d\|_{L^2}^{\frac{6}{q}-1} } dt  \\
        &+C\int_0^T t(\|\nabla u\|_{L^2}^2+\|\nabla^2d\|_{L^2}^2)\|\nabla^2 d\|_{L^2}^4 dt \bigg\} \\
        &\cdot\exp{\left\{ \int_0^T %\bar\rho^{\frac{1}{2}-2\alpha} \|\nabla u\|_{L^2} \|\sqrt{\rho} u_t \|_{L^2}+
        \bar\rho^{-\frac{3}{2}+(\frac{1}{2}-\alpha)\frac{6}{q}}\|\nabla u\|_{L^2}^{3-\frac{6}{q}}\|\sqrt{\rho} u_t\|_{L^2}^{\frac{6}{q}-1}dt\right\}}\\
        &\cdot\exp{\left\{ \int_0^T \|\nabla u\|_{L^2}^2 +\|\nabla^2 d\|_{L^2}^4 + \|\nabla^3 d\|_{L^2}^2 dt\right\}}.
    \end{aligned}
\end{equation}

Taking advantage of \eqref{a1}, \eqref{basic-est} and \eqref{tdu}, we obtain
\begin{equation}\label{111}
    \begin{aligned}
        &\int_0^T t \bar\rho^{ -\alpha}
        %(\bar\rho^{ -\alpha+2(1-\alpha)\frac{6}{q}}+\bar\rho^{ -\frac{1}{2}+\frac{6}{q}(\frac{1}{2}-\alpha)} +{\color{red}\bar\rho^{-\frac{(6+q)\alpha}{q-3}}+\bar\rho^{-2-\frac{6\alpha}{q}}})
        \|\nabla u\|_{L^2}^8 dt\\
        \leq& C %(\bar\rho^{ -\alpha+2(1-\alpha)\frac{6}{q}}
        %+\bar\rho^{ -\frac{1}{2}+\frac{6}{q}(\frac{1}{2}-\alpha)}
        %+{\color{red}\bar\rho^{-\frac{(6+q)\alpha}{q-3}}
        %+\bar\rho^{-2-\frac{6\alpha}{q}}})\\
        \bar\rho^{ -\alpha} \cdot\sup_{t\in[0,T]}\|\nabla u \|_{L^2}^4
        \cdot\sup_{t\in[0,T]}t\|\nabla u \|_{L^2}^2
        \cdot\int_0^T\|\nabla u \|_{L^2}^2dt\\
        \leq& C \bar\rho^{ -\alpha + 2(1-\alpha)}.
        %(\bar\rho^{ -\alpha+2(1-\alpha)\frac{6+q}{q}}+\bar\rho^{ -\frac{1}{2}+\frac{6}{q}(\frac{1}{2}-\alpha)+2(1-\alpha)}+{\color{red}\bar\rho^{-\frac{(6+q)\alpha}{q-3}+2(1-\alpha)}+\bar\rho^{-2-\frac{6\alpha}{q}+2(1-\alpha)}}).
    \end{aligned}
\end{equation}
Using \eqref{dd}, we obtain
\begin{equation}\label{ndt_decay}
\begin{aligned}
    \int_0^T\|\nabla d_t \|_{L^2}^2dt
    \leq &\int_0^T\left(\|\nabla^3 d\|_{L^2}^2
    +\|\nabla(u\cdot\nabla d))\|_{L^2}^2
    +\|\nabla(\vert\nabla d\vert^2 d)\|_{L^2}^2\right)dt\\
    \leq &\int_0^T\bigl(\|\nabla^3 d\|_{L^2}^2
    +\|\nabla u\|_{L^6}^2\|\nabla d\|_{L^3}^2
    +\| u\|_{L^6}^2\|\nabla^2 d\|_{L^3}^2\\
    &\|\nabla d\|_{L^\infty}^3\|\nabla d\|_{L^3}^3
    +\|\nabla d\|_{L^3}^2\|\nabla^2d\|_{L^6}^2\bigr)dt\\
    \leq &\int_0^T \Bigl\{ C(1+\|\nabla u\|_{L^2}^2+\|\nabla d\|_{L^3}^2+\|\nabla d\|_{L^3}^4)\|\nabla^3d\|_{L^2}^2\\
    &+C\bigl(\|\nabla d\|_{L^3}^2\|\sqrt{\rho}u_t\|_{L^2}^2
    +\|\nabla d\|_{L^3}^2\|\nabla u\|_{L^2}^6 \bigr) \Bigr\}dt\\
    \leq & C(\delta)\bar\rho^\alpha.
\end{aligned}
\end{equation}
% Since
% \begin{equation}
% \begin{aligned}
%     \|\nabla^2 d\|_{L^2}^4&\leq C\|\nabla^3 d\|_{L^2}^{\frac{4}{3}}\|\nabla d\|_{L^3}^\frac{8}{3}
%     \leq C\left(\|\nabla d\|_{L^3}^2\|\nabla^3 d\|_{L^2}^2+\|\nabla d\|_{L^3}^4\right)\\
%     &\leq C\left(\|\nabla d\|_{L^3}^2+\|\nabla d\|_{L^2}^2\right)\|\nabla^3 d\|_{L^2}^2\\
%     &\leq C(\vert\Omega\vert)\|\nabla d\|_{L^3}^2\|\nabla^3 d\|_{L^2}^2,
% \end{aligned}
% \end{equation}
% due to
% \begin{equation}\label{GN-d-22}
%    \|\nabla^2 d\|_{L^2}\leq C\|\nabla^3 d\|_{L^2}^{\frac{1}{3}}\|\nabla d\|_{L^3}^\frac{2}{3}.
% \end{equation}
It follows from \eqref{a1} and \eqref{tdu} that
\begin{equation}
    \begin{aligned}
        &\int_0^T t\bar\rho^{-\alpha(1+\frac{6}{q})} \|\nabla u\|_{L^2}^{2}\|\nabla d\|_{L^3}^2\|\nabla ^3d\|_{L^2}^2 dt \\
        %+\color{red}{t\bar\rho^{-2-\frac{6\alpha}{q}}\|\nabla u\|_{L^2}^{9-\frac{6}{q}} \|\nabla^3d\|_{L^2}^{\frac{6}{q}-1} } dt\\
        \le&  C \bar\rho^{-\alpha(1+\frac{6}{q})}  \sup_{t\in[0,T]} \|\nabla u\|_{L^2}^{2} \int_0^T  t \|\nabla ^3d\|_{L^2}^2 dt\\  %+ t\bar\rho^{-2-\frac{6\alpha}{q}} \int_0^T  (\|\nabla u\|_{L^2}^{8} + \|\nabla u\|_{L^2}^{2} \|\nabla^3d\|_{L^2}^2 ) dt\\
        \le & C \bar\rho^{-\alpha(1+\frac{6}{q})}   \|\nabla d_0\|_{L^3}^2, %+ t\bar\rho^{-2-\frac{6\alpha}{q}-\alpha} \int_0^T (\|\nabla u\|_{L^2}^{8} + \|\nabla u\|_{L^2}^{2} \|\nabla^3d\|_{L^2}^2 ) dt\\
    \end{aligned}
\end{equation}
and 
\begin{equation}
    \begin{aligned}
        &\int_0^T t(\|\nabla u\|_{L^2}^2+\|\nabla^2d\|_{L^2}^2)\|\nabla^2 d\|_{L^2}^4 dt \\
        \le & \sup_{0\le t\le T} t(\|\nabla u\|_{L^2}^2+\|\nabla^2d\|_{L^2}^2) \sup_{0\le t\le T} \|\nabla^2d\|_{L^2}^2 \int_0^T \|\nabla^2 d\|_{L^2}^2 dt\\
        \le & C(\bar \rho^{1-\alpha}+ \|\nabla d_0\|_{L^3}^2) \|\nabla d_0\|_{L^3}^2.
    \end{aligned}
\end{equation}

% \begin{equation}\label{113}
%     \begin{aligned}
%         &\bar\rho^{-\alpha}\int_0^Tt\|\nabla^2 d\|_{L^2}^6+t\|\nabla d\|_{L^3}^2(\|\nabla u\|_{L^2}^2+\|\nabla^2d\|_{L^2}^2)\|\nabla^2 d\|_{L^2}^2 dt\\
%         \leq& C\bar\rho^{-\alpha}\int_0^T t(\|\nabla^2 d\|_{L^2}^4+\|\nabla^2 d\|_{L^2}^6+\|\nabla u\|_{L^2}^2\|\nabla^2 d\|_{L^2}^2) dt\\
%         \leq& C\bar\rho^{-\alpha}\int_0^T t\left(\|\nabla d\|_{L^3}^2+\|\nabla d\|_{L^3}^4\right)\|\nabla^3 d\|_{L^2}^2dt
%         +C\bar\rho^{-\alpha} \sup_{t\in[0,T]}t\|\nabla^2 d \|_{L^2}^2
%         \cdot\int_0^T\|\nabla u \|_{L^2}^2dt\\
%         \leq & C(\bar\rho^{1-\alpha}+\bar\rho^{2-2\alpha}),
%     \end{aligned}
% \end{equation}
% and
% \begin{equation}
%     \begin{aligned}
%         &\int_0^T t\bar\rho^{-\alpha}\|\nabla d\|_{L^3}^2\|\nabla ^3d\|_{L^2}^2
%         +Ct\bar\rho^{-2-\frac{6\alpha}{q}}\|\nabla d\|_{L^3}^6\|\nabla u\|_{L^2}^6\|\nabla ^3d\|_{L^2}^2 dt\\
%         \leq& C(\delta^2\bar\rho^{-\alpha}+\delta^6\bar\rho^{-2-\frac{6\alpha}{q}})\int_0^T t\|\nabla ^3d\|_{L^2}^2 dt\\
%         \leq & C\delta^2(\bar\rho^{1-\alpha}+\bar\rho^{-1-\frac{6\alpha}{q}}).
%     \end{aligned}
% \end{equation}
H\"older's inequality yields
% \begin{equation}
%     \begin{aligned}
%         &\int_0^T \bar\rho^{\frac{1}{2}-2\alpha} \|\nabla u\|_{L^2} \|\sqrt{\rho} u_t \|_{L^2} dt\\
%         \leq & \bar\rho^{\frac{1}{2}-2\alpha}
%         \left(\int_0^T\|\nabla u\|_{L^2}^2dt\right)^\frac{1}{2}
%         \left(\int_0^T\|\sqrt{\rho} u_t\|_{L^2}^2dt\right)^\frac{1}{2}\\
%         \leq& C\bar\rho^{1-2\alpha},
%     \end{aligned}
% \end{equation}
% and
\begin{equation}
    \begin{aligned}
        &\int_0^T \bar\rho^{-\frac{3}{2}+(\frac{1}{2}-\alpha)\frac{6}{q}}\|\nabla u\|_{L^2}^{3-\frac{6}{q}}\|\sqrt{\rho} u_t\|_{L^2}^{\frac{6}{q}-1} dt\\
        \leq& \bar\rho^{-\frac{3}{2}+(\frac{1}{2}-\alpha)\frac{6}{q}}
        \left(\int_0^T\|\nabla u\|_{L^2}^2dt\right)^\frac{3q-6}{2q}
        \left(\int_0^T\|\sqrt{\rho} u_t\|_{L^2}^2dt\right)^\frac{6-q}{2q}\\
        \leq& C\bar\rho^{-2\alpha}.
    \end{aligned}
\end{equation}
The last term
% \begin{equation}
%     \begin{aligned}
%         &\bar\rho^{-2-\frac{6\alpha}{q}}\int_0^T \left( {\color{red} \|\nabla u\|_{L^2}^{3-\frac{6}{q}} \|\nabla^3d\|_{L^2}^{\frac{6}{q}-1} }\right)dt\\
%         \leq &\bar\rho^{-2-\frac{6\alpha}{q}}\int_o^T C\bigl(\delta^2\|\nabla u\|_{L^2}^2
%         +\delta^4\|\nabla ^3d\|_{L^2}^2\bigr)dt\\
%         \leq &C\left(\delta^2\bar\rho^{-1-(1+\frac{6}{q})\alpha}
%         +\delta^4\bar\rho^{-2-(\frac{6}{q}-1)\alpha}\right)\leq \Tilde{C}\delta^2,
%     \end{aligned}
% \end{equation}
%and
\begin{equation}\label{113}
    \begin{aligned}
        &\int_0^T \left(
        \| \nabla u\|_{L^2}^2
        +\|\nabla^2 d\|_{L^2}^4+ \|\nabla^3 d\|_{L^2}^2\right)dt\\
        \leq & C \int_0^T \bigl(\|\nabla u\|_{L^2}^2
        +\|\nabla d\|_{L^3}^2\|\nabla ^3d\|_{L^2}^2 + \|\nabla^3 d\|_{L^2}^2\bigr)dt\\
        \leq &C\left(\bar\rho^{1-\alpha}
        + c_4 \|\nabla^2 d_0\|_{L^2}^2\right)\leq \Tilde{C} ,
    \end{aligned}
\end{equation}
due to $\delta<1<\bar\rho$ .
Hence, collecting all the estimates \eqref{111}-\eqref{113}, one gets
\begin{equation}\label{trut1}
\begin{gathered}
    \sup_{0\le t\le T} t \left(\|\sqrt{\rho} u_t \|_{L^2}^2
    + \|\nabla d_t\|_{L^2}^2\right) 
    +\int_{0}^{T} t\left(\bar\rho^\alpha\|\nabla u_t\|^2_{L^2}
    + \|\Delta d_t\|_{L^2}^2\right) dt\\
    \leq  C(\bar\rho^{-A}+\bar\rho^\alpha)\cdot
    \exp{\{C\bar\rho^{-B}+\Tilde{C}\}}
    \leq C\bar\rho^\alpha.
\end{gathered}
\end{equation}
with some $A,B>0$.

On the other hand, multiplying \eqref{k3} by $t$, one has
\begin{equation}
    \begin{aligned}
    &\frac{d}{dt}(t^2\|\sqrt{\rho}u_t\|_{L^2}^2
    +t^2 \|\nabla d_t\|_{L^2}^2)
    +t^2(\bar\rho^\alpha\|\nabla u_t\|^2
    + \|\nabla^2 d_t\|_{L^2}^2)\\
    &\leq Ct\|\sqrt{\rho}u_t\|_{L^2}^2
    +Ct \|\nabla d_t\|_{L^2}^2 + Ct^2 \|\nabla u\|_{L^2}^2 \|\sqrt{\rho} u_t \|_{L^2}^2 + Ct^2\rho^{ -\alpha} \|\nabla u\|_{L^2}^8\\
        &+ Ct^2 \bar\rho^{-\frac{3}{2}+(\frac{1}{2}-\alpha)\frac{6}{q}}\|\nabla u\|_{L^2}^{3-\frac{6}{q}}\|\sqrt{\rho} u_t\|_{L^2}^{1+\frac{6}{q}} \\%+ c_{10} t^2\bar\rho^{-\alpha} \left( \|\nabla^2 d_t\|_{L^2}^2 +  \bar \rho^{\alpha} \|\nabla u_t\|_{L^2}^2\right)\\
        &+Ct^2 \bar\rho^{-\alpha(1+\frac{6}{q})} \|\nabla u\|_{L^2}^{2} \|\nabla d\|_{L^3}^2\|\nabla ^3d\|_{L^2}^2\\
        %+\color{red}{Ct^2 \bar\rho^{-2-\frac{6\alpha}{q}}\|\nabla u\|_{L^2}^{9-\frac{6}{q}} \|\nabla^3d\|_{L^2}^{\frac{6}{q}-1} } \\
        & +Ct^2(\|\nabla u\|_{L^2}^4+\|\nabla^2 d\|_{L^2}^4 + \|\nabla u\|_{L^2}^{2} + \|\nabla^3d\|_{L^2}^{2} )\|\nabla d_t\|_{L^2}^2 \\
        & + Ct^2 (\|\nabla u\|_{L^2}^2+\|\nabla^2d\|_{L^2}^2)\|\nabla^2 d\|_{L^2}^4,
    % %+{\color{red}Ct^2\bar\rho^{-2\alpha}\|\sqrt{\rho}u_t\|_{L^2}^4}\\
    % &+Ct^2(\rho^{ -\alpha+2(1-\alpha)\frac{6}{q}}+\bar\rho^{ -\frac{1}{2}+\frac{6}{q}(\frac{1}{2}-\alpha)}
    % +{\color{red}\bar\rho^{-\frac{(6+q)\alpha}{q-3}}
    % +\bar\rho^{-2-\frac{6\alpha}{q}}})
    % \|\nabla u\|_{L^2}^8\\
    % &+Ct^2(\bar\rho^{2-3\alpha}+\bar\rho^{ -\frac{1}{2}+\frac{6}{q}(\frac{1}{2}-\alpha)}) \|\nabla u\|_{L^2}^2 \|\sqrt{\rho} u_t \|_{L^2}^2\\
    % &+Ct^2\bar\rho^{\frac{1}{2}-2\alpha} \|\nabla u\|_{L^2} \|\sqrt{\rho} u_t \|_{L^2}^3+Ct \bar\rho^{-\frac{3}{2}+(\frac{1}{2}-\alpha)\frac{6}{q}}\|\nabla u\|_{L^2}^{3-\frac{6}{q}}\|\sqrt{\rho} u_t\|_{L^2}^{1+\frac{6}{q}}\\
    % &+Ct^2\bar\rho^{-\alpha}\|\nabla^2 d\|_{L^2}^6
    % +Ct^2\bar\rho^{-\alpha}\|\nabla d\|_{L^3}^2(\|\nabla u\|_{L^2}^2
    % +\|\nabla^2d\|_{L^2}^2)\|\nabla^2 d\|_{L^2}^2\\
    % &+Ct^2\bar\rho^{-2\alpha}\|\nabla d\|_{L^3}^2\|\nabla ^3d\|_{L^2}^2
    % +Ct^2\bar\rho^{-2-\frac{6\alpha}{q}}\|\nabla d\|_{L^3}^6\|\nabla u\|_{L^2}^6\|\nabla ^3d\|_{L^2}^2\\
    % &+Ct^2{\color{red}\bar\rho^{-2-\frac{6\alpha}{q}}(\|\nabla d\|_{L^3}^2\|\nabla u\|_{L^2}^2}
    % +\|\nabla d\|_{L^3}^4\|\nabla ^3d\|_{L^2}^2)
    % \| \nabla d_t\|_{L^2}^2\\
    % &+Ct^2\bar\rho^{-\alpha}(\|\nabla u\|_{L^2}^4
    % +\|\nabla^2 d\|_{L^2}^4)\| \nabla d_t\|_{L^2}^2.\\
\end{aligned}
\end{equation}
Applying Gronwall’s inequality to arrive at
\begin{equation}
    \begin{aligned}
        &\sup_{0\le t\le T} t^2 \left(\|\sqrt{\rho} u_t \|_{L^2}^2
        + \|\nabla d_t\|_{L^2}^2\right) 
        +\int_{0}^{T} t^2\left(\bar\rho^\alpha\|\nabla u_t\|^2_{L^2}
        + \| \Delta d_t\|_{L^2}^2\right) dt\\
        \leq &  \bigg\{C\int_0^T t^2 \bar\rho^{ -\alpha } 
        \|\nabla u\|_{L^2}^8  dt 
        +C\int_0^T t\left(\|\sqrt{\rho} u_t\|_{L^2}^2
        + \|\nabla d_t\|_{L^2}^2 \right)dt\\
        &+C\int_0^T t^2\bar\rho^{-\alpha(1+\frac{6}{q})} \|\nabla u\|_{L^2}^{2} \|\nabla d\|_{L^3}^2\|\nabla ^3d\|_{L^2}^2 dt\\
        %+\color{red}{t^2 \bar\rho^{-2-\frac{6\alpha}{q}}\|\nabla u\|_{L^2}^{9-\frac{6}{q}} \|\nabla^3d\|_{L^2}^{\frac{6}{q}-1} } dt  \\
        &+C\int_0^T t^2 (\|\nabla u\|_{L^2}^2+\|\nabla^2d\|_{L^2}^2)\|\nabla^2 d\|_{L^2}^4 dt \bigg\} \\
        &\cdot\exp{\{\bar\rho^{-B}+\Tilde{C}\}}.
    \end{aligned}
\end{equation}
\eqref{tdu} together with \eqref{trut} yields
\begin{equation}
    \begin{aligned}
        &\int_0^T t^2\bar\rho^{ -\alpha}
        \|\nabla u\|_{L^2}^8 dt\\
        \leq& C \bar\rho^{ -\alpha} \sup_{t\in[0,T]}\|\nabla u \|_{L^2}^2
        \cdot\sup_{t\in[0,T]}t^2\|\nabla u \|_{L^2}^4
        \cdot\int_0^T\|\nabla u \|_{L^2}^2dt\\
        \leq& C \bar\rho^{ -\alpha+3(1-\alpha)}.
    \end{aligned}
\end{equation}
Similar to \eqref{ndt_decay}, we have
\begin{equation}
\begin{aligned}
    \int_0^T t\|\nabla d_t \|_{L^2}^2dt
    \leq &\int_0^T \Bigl\{ Ct(1+\|\nabla u\|_{L^2}^2+\|\nabla d\|_{L^3}^2+\|\nabla d\|_{L^3}^4)\|\nabla^3d\|_{L^2}^2\\
    &+Ct\bigl(\|\nabla d\|_{L^3}^2\|\sqrt{\rho}u_t\|_{L^2}^2
    +\|\nabla d\|_{L^3}^2\|\nabla u\|_{L^2}^6 \bigr) \Bigr \} dt\\
    \leq & C(\delta)\bar\rho.
\end{aligned}
\end{equation}
And it follows from \eqref{tdu} and \eqref{tdd} that
\begin{equation}
    \begin{aligned}
        &\int_0^T  t^2 \bar\rho^{-\alpha(1+\frac{6}{q})} \|\nabla u\|_{L^2}^{2} \|\nabla d\|_{L^3}^2\|\nabla ^3d\|_{L^2}^2 dt \\
        %+\color{red}{t^2\bar\rho^{-2-\frac{6\alpha}{q}}\|\nabla u\|_{L^2}^{9-\frac{6}{q}} \|\nabla^3d\|_{L^2}^{\frac{6}{q}-1} } dt\\
        \le& C \bar\rho^{-\alpha(1+\frac{6}{q})} \sup_{0\le t\le T} t \|\nabla u\|_{L^2}^2 \int_0^T t \|\nabla^3 d\|_{L^2}^2 dt\\
        \le& C \bar\rho^{-\alpha(1+\frac{6}{q})+1-\alpha},
    \end{aligned}
\end{equation}
and 
\begin{equation}
    \begin{aligned}
        &\int_0^T t^2(\|\nabla u\|_{L^2}^2+\|\nabla^2d\|_{L^2}^2)\|\nabla^2 d\|_{L^2}^4 dt \\
        \le & \sup_{0\le t\le T} t(\|\nabla u\|_{L^2}^2+\|\nabla^2d\|_{L^2}^2) \sup_{0\le t\le T} t \|\nabla^2d\|_{L^2}^2 \int_0^T \|\nabla^2 d\|_{L^2}^2 dt\\
        \le & C(\bar \rho^{1-\alpha}+ \|\nabla d_0\|_{L^3}^2) \|\nabla d_0\|_{L^3}^2.
    \end{aligned}
\end{equation}

% Note that Poincar\'e inequality leads to
% \begin{equation}
%     \|\nabla d\|_{L^3}^2\leq C \|\nabla d\|_{L^2}\|\nabla^2 d\|_{L^2}\leq C\|\nabla^2 d\|_{L^2}^2,
% \end{equation}
% which together with \eqref{a1}, \eqref{a2} and \eqref{tdu} yields
% \begin{equation}\label{114}
%     \begin{aligned}
%         &\bar\rho^{-\alpha}\int_0^T t^2\|\nabla^2 d\|_{L^2}^6
%         +t^2\|\nabla d\|_{L^3}^2(\|\nabla u\|_{L^2}^2
%         +\|\nabla^2d\|_{L^2}^2)\|\nabla^2 d\|_{L^2}^2 dt\\
%         \leq& C\bar\rho^{-\alpha}\int_0^T t^2\left( (\|\nabla d\|_{L^3}^2+\|\nabla d\|_{L^3}^4)\|\nabla^3 d\|_{L^2}^2+\|\nabla u\|_{L^2}^2\|\nabla^2 d\|_{L^2}^4\right) dt\\
%         \leq& C\bar\rho^{-\alpha} \sup_{t\in[0,T]}t\|\nabla^2 d \|_{L^2}^2
%         \cdot\int_0^T t\|\nabla^3 d \|_{L^2}^2dt
%         +C\bar\rho^{-\alpha} \sup_{t\in[0,T]}t^2\|\nabla^2 d \|_{L^2}^4
%         \cdot\int_0^T\|\nabla u \|_{L^2}^2dt\\
%         \leq & C(\bar\rho^{2-2\alpha}+\bar\rho^{3-2\alpha})
%         \leq C\bar\rho,
%     \end{aligned}
% \end{equation}
% and
% \begin{equation}
%     \begin{aligned}
%         &\int_0^T t^2\bar\rho^{-\alpha}\|\nabla d\|_{L^3}^2\|\nabla ^3d\|_{L^2}^2
%         +Ct^2\bar\rho^{-2-\frac{6\alpha}{q}}\|\nabla d\|_{L^3}^6\|\nabla u\|_{L^2}^6\|\nabla ^3d\|_{L^2}^2 dt\\
%         \leq& C(\bar\rho^{-\alpha}
%         +\bar\rho^{-2-\frac{6\alpha}{q}})\sup_{t\in[0,T]}t\|\nabla^2 d \|_{L^2}^2
%         \cdot\int_0^T t\|\nabla ^3d\|_{L^2}^2 dt\\
%         \leq & C(\bar\rho^{2-\alpha}+\bar\rho^{-\frac{6\alpha}{q}})
%         \leq C\bar\rho,
%     \end{aligned}
% \end{equation}
% due to $\alpha>1$.

Hence it follows immediately by \eqref{trut1} that,
\begin{equation}
\begin{gathered}
    \sup_{0\le t\le T} t^2 \left(\|\sqrt{\rho} u_t \|_{L^2}^2+\|\nabla d_t\|_{L^2}^2\right) 
        +\int_{0}^{T} t^2\left(\bar\rho^\alpha\|\nabla u_t\|^2_{L^2}+\|\Delta d_t\|_{L^2}^2\right) dt\\
    \leq C(\bar\rho^{-A_1}+\bar\rho)\cdot
    \exp{\{\bar\rho^{-B_1}+\Tilde{C}\}}
    \leq C\bar\rho,
\end{gathered}
\end{equation}
with some $A_1,B_1>0$.
\end{proof}
Finally, we are about to finish the bound of  $\mathcal{E}_\rho$, the key observation is that $\|\nabla u\|_{L^1_tL^\infty_x}$ is uniformly bounded with respect to time $T$.
\begin{lemma}\label{L_1}
	There exists a positive constant $\Lambda_2$ such that 
	\begin{equation}
		\mathcal{E}_\rho(T)\le  2 \mathcal{E}_\rho(0),
	\end{equation}
	provided $\bar\rho>\Lambda_2=\Lambda_2(\Omega,C_0,\mu,\nu,\lambda, \alpha,\|\nabla\rho_0\|_{L^q}, \| u_0\|_{H^2}, \|\nabla d_0\|_{H^1})$.
\end{lemma}
\begin{proof}
	It follows from \eqref{ins}$_1$  that  
	\begin{equation}
		\nabla \rho_t + u \cdot \nabla^2 \rho + \nabla u \cdot \nabla \rho =0.
	\end{equation}
	Multiplying the above equation by $|\nabla \rho|^{q-2}\nabla \rho$ and then integrating by parts, we have 
	\begin{equation}\label{dr}
		\begin{aligned}
			\frac{1}{p}\frac{d}{dt} \|\nabla \rho\|_{L^q}^q  = & - \int \nabla \rho \cdot \nabla u \cdot \nabla \rho |\nabla \rho|^{q-2} dx  \\
			\le & C \|\nabla u\|_{L^\infty} \|\nabla \rho \|_{L^q}^q .
		\end{aligned}
	\end{equation}
It follows from \eqref{W2q} and Gagliardo-Nirenber inequality that

	\begin{equation}\label{kkk}
		\begin{aligned}
			\int_0^T \|\nabla u\|_{L^\infty} dt
   \le & \int_0^T \|\nabla u\|_{W^{1,q}} dt\\
   \le & C \bar\rho^{-\alpha} \int_0^T  \|\rho u_t\|_{L^q}   dt + \bar \rho^{(1-\alpha)\frac{5q-6}{q}} \int_0^T  \|\nabla u\|_{L^2}^{\frac{6(q-1)}{q}} dt \\
   &+ C \bar\rho^{-\alpha}\int_0^T  \|\nabla d\|_{L^3}^{\frac{2}{q}}\|\nabla^3 d\|_{L^2}^{2-\frac{2}{q}}dt.
		\end{aligned}
	\end{equation}
For the first term on the right-hand side of the above inequality, after using the Gagliardo-Nirenberg inequality and \eqref{trut} and \eqref{ttrut}, we have 
    \begin{equation}\label{qdecay}
        \begin{aligned}
            &\int_{0}^{T}  \|\rho u_t\|_{L^q} dt \\
            \le & \int_{0}^{T}  \bar\rho^{\frac{5q-6}{4q}}\|\sqrt{\rho}u_t\|_{L^2}^{\frac{6-q}{2q}} \|\nabla u_t\|_{L^2}^{\frac{3(q-2)}{2q}} dt\\
            \le & C \bar\rho^{\frac{5q-6}{4q}} \left(\sup_{0\le t\le \min\{1,T\}} t \|\sqrt{\rho}u_t\|_{L^2}^2 dt\right)^{\frac{6-q}{4q}}\\
            &\cdot \left( \int_{0}^{\min\{1,T\}} t \|\nabla u_t\|_{L^2}^2 dt \right)^{\frac{3(q-2)}{4q}}\left( \int_{0}^{\min\{1,T\}} t^{-\frac{2q}{q+6}}  dt \right)^{\frac{q+6}{4q}} \\
            &+ C \bar\rho^{\frac{5q-6}{4q}} \left(\sup_{\min\{1,T\}\le t\le T} t^2 \|\sqrt{\rho}u_t\|_{L^2}^2 dt\right)^{\frac{6-q}{4q}} \\
            &\cdot\left( \int_{\min\{1,T\}}^{T} t^2 \|\nabla u_t\|_{L^2}^2 dt \right)^{\frac{3(q-2)}{4q}} \left( \int_{\min\{1,T\}}^{T} t^{-\frac{4q}{q+6}}  dt \right)^{\frac{q+6}{4q}} \\
            \le & C \bar\rho^{\frac{5q-6}{4q} + \alpha \frac{6-q}{4q} } + C \bar\rho^{\frac{5q-6}{4q} + \frac{6-q}{4q} + (1-\alpha) \frac{3(q-2)}{4q}} \\
            \le & C \bar\rho^{\frac{5q-6}{4q} + \alpha \frac{6-q}{4q}}.
        \end{aligned}
    \end{equation}
The last term
\begin{equation}\label{ddecay}
\begin{aligned}
	    &\bar\rho^{-\alpha}\int_0^T  \|\nabla d\|_{L^3}^{\frac{2}{q}}\|\nabla^3 d\|_{L^2}^{2-\frac{2}{q}}dt
        \leq C\delta^\frac{2}{q}\bar\rho^{-\alpha}\int_0^T  \|\nabla^3 d\|_{L^2}^{2-\frac{2}{q}}dt\\
        \le & \bar\rho^{-\alpha}
        \left(\int_0^{\min\{1,T\}} \|\nabla^3 d\|_{L^2}^2dt\right)^{1-\frac{1}{q}}\\
        &+ \bar\rho^{-\alpha}
        \left(\int_{\min\{1,T\}} ^Tt\|\nabla^3 d\|_{L^2}^2dt\right)^{1-\frac{1}{q}}\left(\int_{\min\{1,T\}} ^T t^{1-q}dt\right)^\frac{1}{q}\\
        \leq & C \bar\rho^{-\alpha}.
\end{aligned}
\end{equation}
\eqref{ddecay} together with \eqref{qdecay}, \eqref{kkk}, yields that after using \eqref{basic-est},
    \begin{equation}
		\begin{aligned}
		\int_0^T \|\nabla u\|_{L^\infty} dt 
            \le & C \bar\rho^{\frac{5q-6}{4q} + \alpha \frac{6-q}{4q} -\alpha } +  C \bar \rho^{(1-\alpha)\frac{5q-6}{q} + (1-\alpha)}
            +C\delta^\frac{2}{q}\bar\rho^{-\alpha}\\
            \le & C \bar\rho^{\frac{5q-6}{4q}(1 -\alpha) } +  C \bar \rho^{(1-\alpha)\frac{5q-6}{q} + (1-\alpha)}
            +C\bar\rho^{-\alpha}\\
            \leq & C \bar\rho^{-D},
		\end{aligned}
	\end{equation}
 where
 \begin{equation*}
     D=\max{\left\{\frac{5q-6}{4q}(\alpha-1), \frac{6(q-1)}{q}(\alpha-1),-\alpha\right\}}.
 \end{equation*}
    Finally, note
    \begin{equation}
        \mathcal{E}_\rho(0)= \|\nabla \rho_0\|_{L^q},
    \end{equation}
    then Gronwall’s inequality together with \eqref{dr} yields
    \begin{equation}
        \sup_{0\le t\le T} \|\nabla \rho\|_{L^q}  
        \le \exp{\left\{C_2\bar\rho^{-D}\right\}}\|\nabla \rho_0\|_{L^q}\leq 2\mathcal{E}_\rho(0),
    \end{equation}
    provided 
    \begin{equation}
        \bar{\rho}\geq \biggl(\frac{C_2}{\log 2}\biggr)^\frac{1}{D}\triangleq \Lambda_2.
    \end{equation}
\end{proof}
\textbf{Proof of Prosition \ref{pr}}
    Proposition \ref{pr} is a direct consequence of Lemmas \ref{3d}, \ref{L_2} and \ref{L_1} after choosing
    \begin{equation}\label{delta}
        \delta < \min\left\{\frac{1}{4c_1}, \frac{1}{4c_3}, \frac{1}{4c_8}, \frac{1}{4c_9}\right\},
    \end{equation}
    with constants $c_1,c_3,c_8,c_9$ depending only on the Sobolev constants and elliptic constants,
    and 
    $$\Lambda_0=\max\left\{\Lambda_1, \Lambda_2, \frac{1}{(4c_{10})^{\alpha}} \right\}.$$
    
\section{Proof of Theorem \ref{global}}
According to Theorem \ref{local}, there exists a $\Tilde{T}>0$ such that the inhomogeneous incompressible simplified Ericksen-Leslie system \eqref{ins}-\eqref{bc} has a unique local strong solution $(\rho, u, P, d)$ on $[0, \Tilde{T}]$. We use the a priori estimates, Proposition \ref{pr} to extend the local strong solution to all time.

Due to 
\begin{equation}
    \|\nabla \rho_0\|_{L^q}=\mathcal{E}_\rho(0)<3\mathcal{E}_\rho(0), \ 
    \|\nabla u_0\|_{L^2}^2=M<3M,\
    \|\nabla d_0\|_{L^3}=\mathcal{E}_d(0)<2\delta,
\end{equation}
and the local regularity results \eqref{l-r}, there exists a $T_1\in(0, \Tilde{T})$ such that
\begin{equation}
    \sup_{0\le t\le T_1} \|\nabla \rho\|_{L^q}\leq 3\mathcal{E}_\rho(0),\ 
    \sup_{0\le t\le T_1} \|\nabla u_0\|_{L^2}^2 \leq 3M,\
    \sup_{0\le t\le T_1} \|\nabla d_0\|_{L^3} \leq 2\delta.
\end{equation}
Set
\begin{equation}
    T^\ast=\sup\{T| (\rho, u, P, d)\ \mathrm{is}\ \mathrm{a}\ \mathrm{strong}\ \mathrm{solution}\ \mathrm{to}\ \eqref{ins}-\eqref{bc}\ \mathrm{on}\ [0,T]\},
\end{equation}
\begin{equation}
T_1^\ast=\sup\left\{T\bigg| 
\begin{array}{l}
(\rho, u, P, d)\ \mathrm{is}\ \mathrm{a}\ \mathrm{strong}\ \mathrm{solution}\ \mathrm{to}\ \eqref{ins}-\eqref{bc}\ \mathrm{on}\ [0,T],\\
\sup_{0\le t\le T_1} \|\nabla \rho\|_{L^q}\leq 3\mathcal{E}_\rho(0),\ 
\sup_{0\le t\le T_1} \|\nabla u_0\|_{L^2}^2 \leq 3M,\\
\sup_{0\le t\le T_1} \|\nabla d_0\|_{L^3} \leq 2\delta.
\end{array} \right\}.
\end{equation}
Then $T^\ast_1\geq T_1>0$. Recalling Proposition \ref{pr}, it’s easy to verify
\begin{equation}
    T^\ast=T^\ast_1.
\end{equation}
provided that $\bar\rho>\Lambda_0$ and $\|\nabla d_0\|_{L^3}\leq  \varepsilon_0$ as assumed. 

We claim that $T^\ast=\infty$. Otherwise, assume that
$T^\ast<\infty$. By virtue of Proposition \ref{pr}, for every $t\in[0, T^\ast)$, it holds that
\begin{equation}
     \|\nabla \rho\|_{L^q}\leq 2\mathcal{E}_\rho(0),\ 
     \|\nabla u_0\|_{L^2}^2 \leq 2M,\
     \|\nabla d_0\|_{L^3} \leq \delta,
\end{equation}
therefore we can extend the solution to $T^{\ast\ast}>T^\ast$ due to Lemma \ref{local}, which contradicts with the defination of $T^\ast$. Hence we finish the proof of Theorem \ref{global}.

\section*{Conflict-of-interest statement}
All authors declare that they have no conflicts of interest.

\section*{Data Availability}
No data were used for the research described in the article.

\section*{Acknowledgments}
Y. Mei is supported by the National Natural Science Foundation of China No. 12101496 and 12371227.	
J.-X. Li was supported in part by Zheng Ge Ru Foundation, 
%Hong Kong RGC Earmarked Research Grants CUHK-14301421, CUHK-14300819, CUHK-14302819, CUHK-14300917, the key project of NSFC (Grant No. 12131010)  
the Shun Hing Education and Charity Fund. 
R. Zhang is supported by the National Natural Science Foundation of China No. 12401279. 
Part of this work was done when Y. Mei and R. Zhang were visiting the Institute of Mathematical Sciences at the Chinese University of Hong Kong. They would like to thank the institute for its hospitality.

\normalem
\bibliographystyle{siam}
\bibliography{ref}

\begin{thebibliography}{10}

\bibitem{abidi2015global1}
{\sc H.~Abidi and P.~Zhang}, {\em Global well-posedness of 3-d
  density-dependent navier-stokes system with variable viscosity}, Science
  China Mathematics, 58 (2015), pp.~1129--1150.

\bibitem{abidi2015global}
\leavevmode\vrule height 2pt depth -1.6pt width 23pt, {\em On the global
  well-posedness of 2-d inhomogeneous incompressible navier--stokes system with
  variable viscous coefficient}, Journal of Differential Equations, 259 (2015),
  pp.~3755--3802.

\bibitem{agmon1964estimates}
{\sc S.~Agmon, A.~Douglis, and L.~Nirenberg}, {\em Estimates near the boundary
  for solutions of elliptic partial differential equations satisfying general
  boundary conditions ii}, Communications on pure and applied mathematics, 17
  (1964), pp.~35--92.

\bibitem{cho2004unique}
{\sc Y.~Cho and H.~Kim}, {\em Unique solvability for the density-dependent
  navier--stokes equations}, Nonlinear Analysis: Theory, Methods \&
  Applications, 59 (2004), pp.~465--489.

\bibitem{craig2013global}
{\sc W.~Craig, X.~Huang, and Y.~Wang}, {\em Global wellposedness for the 3d
  inhomogeneous incompressible navier--stokes equations}, Journal of
  Mathematical Fluid Mechanics, 15 (2013), pp.~747--758.

\bibitem{desjardins1997regularity}
{\sc B.~Desjardins}, {\em Regularity results for two-dimensional flows of
  multiphase viscous fluids}, Archive for Rational Mechanics and Analysis, 137
  (1997), pp.~135--158.

\bibitem{ericksen1961conservation}
{\sc J.~L. Ericksen}, {\em Conservation laws for liquid crystals}, Transactions
  of the Society of Rheology, 5 (1961), pp.~23--34.

\bibitem{frank1958liquid}
{\sc F.~C. Frank}, {\em I. liquid crystals. on the theory of liquid crystals},
  Discussions of the Faraday Society, 25 (1958), pp.~19--28.

\bibitem{fujita1964navier}
{\sc H.~Fujita and T.~Kato}, {\em On the navier-stokes initial value problem.
  i}, Archive for rational mechanics and analysis, 16 (1964), pp.~269--315.

\bibitem{galdi2011introduction}
{\sc G.~Galdi}, {\em An introduction to the mathematical theory of the
  Navier-Stokes equations: Steady-state problems}, Springer Science \& Business
  Media, 2011.

\bibitem{gao2016strong}
{\sc J.~Gao, Q.~Tao, and Z.-a. Yao}, {\em Strong solutions to the
  density-dependent incompressible nematic liquid crystal flows}, Journal of
  Differential Equations, 260 (2016), pp.~3691--3748.

\bibitem{gong2016local}
{\sc H.~Gong, J.~Li, and C.~Xu}, {\em Local well-posedness of strong solutions
  to density-dependent liquid crystal system}, Nonlinear Analysis: Theory,
  Methods \& Applications, 147 (2016), pp.~26--44.

\bibitem{hao2024global}
{\sc T.~Hao, F.~Shao, D.~Wei, and Z.~Zhang}, {\em Global well-posedness of
  inhomogeneous navier-stokes equations with bounded density}, arXiv preprint
  arXiv:2406.19907,  (2024).

\bibitem{he2021global}
{\sc C.~He, J.~Li, and B.~L{\"u}}, {\em Global well-posedness and exponential
  stability of 3d navier--stokes equations with density-dependent viscosity and
  vacuum in unbounded domains}, Archive for Rational Mechanics and Analysis,
  239 (2021), pp.~1809--1835.

\bibitem{hieber2016dynamics}
{\sc M.~Hieber, M.~Nesensohn, J.~Pr{\"u}ss, and K.~Schade}, {\em Dynamics of
  nematic liquid crystal flows: The quasilinear approach}, Annales de
  l'Institut Henri Poincar{\'e} C, Analyse non lin{\'e}aire, 33 (2016),
  pp.~397--408.

\bibitem{hineman2013well}
{\sc J.~L. Hineman and C.~Wang}, {\em Well-posedness of nematic liquid crystal
  flow in $l^3_{uloc}(\mathbb{R}^3)$}, Archive for Rational Mechanics and
  Analysis, 210 (2013), pp.~177--218.

\bibitem{hoff1995global}
{\sc D.~Hoff}, {\em Global solutions of the navier-stokes equations for
  multidimensional compressible flow with discontinuous initial data}, Journal
  of Differential Equations, 120 (1995), pp.~215--254.

\bibitem{hong2011global}
{\sc M.-C. Hong}, {\em Global existence of solutions of the simplified
  ericksen--leslie system in dimension two}, Calculus of Variations and Partial
  Differential Equations, 40 (2011), pp.~15--36.

\bibitem{hong2014blow}
{\sc M.-C. Hong, J.~Li, and Z.~Xin}, {\em Blow-up criteria of strong solutions
  to the ericksen-leslie system in $\mathbb{R}^3$}, Communications in Partial
  Differential Equations, 39 (2014), pp.~1284--1328.

\bibitem{hong2019well}
{\sc M.-C. Hong and Y.~Mei}, {\em Well-posedness of the ericksen--leslie system
  with the oseen--frank energy in l \_ uloc\^{} 3 ($\backslash$mathbb r\^{} 3)
  l uloc 3 (r 3)}, Calculus of Variations and Partial Differential Equations,
  58 (2019), pp.~1--38.

\bibitem{hong2012global}
{\sc M.-C. Hong and Z.~Xin}, {\em Global existence of solutions of the liquid
  crystal flow for the oseen--frank model in $\mathbb{R}^2$}, Advances in
  Mathematics, 231 (2012), pp.~1364--1400.

\bibitem{hu2018global}
{\sc X.~Hu and Q.~Liu}, {\em Global solution to the 3d inhomogeneous nematic
  liquid crystal flows with variable density}, Journal of Differential
  Equations, 264 (2018), pp.~5300--5332.

\bibitem{huang2014regularity}
{\sc J.~Huang, F.~Lin, and C.~Wang}, {\em Regularity and existence of global
  solutions to the ericksen--leslie system in r\^{} 2 r 2}, Communications in
  Mathematical Physics, 331 (2014), pp.~805--850.

\bibitem{huang2016finite}
{\sc T.~Huang, F.~Lin, C.~Liu, and C.~Wang}, {\em Finite time singularity of
  the nematic liquid crystal flow in dimension three}, Archive for Rational
  Mechanics and Analysis, 221 (2016), pp.~1223--1254.

\bibitem{huang2012blow}
{\sc T.~Huang and C.~Wang}, {\em Blow up criterion for nematic liquid crystal
  flows}, Communications in Partial Differential Equations, 37 (2012),
  pp.~875--884.

\bibitem{huang2024global}
{\sc X.~Huang, J.~Li, and R.~Zhang}, {\em Global large strong solution of the
  3d inhomogeneous navier-stokes equations with density-dependent viscosity},
  arXiv preprint arXiv:2408.00333,  (2024).

\bibitem{huang2014global}
{\sc X.~Huang and Y.~Wang}, {\em Global strong solution with vacuum to the two
  dimensional density-dependent navier--stokes system}, SIAM Journal on
  Mathematical Analysis, 46 (2014), pp.~1771--1788.

\bibitem{huang2015global}
\leavevmode\vrule height 2pt depth -1.6pt width 23pt, {\em Global strong
  solution of 3d inhomogeneous navier--stokes equations with density-dependent
  viscosity}, Journal of Differential Equations, 259 (2015), pp.~1606--1627.

\bibitem{ladyzhenskaia1968linear}
{\sc O.~A. Ladyzhenskaia, V.~A. Solonnikov, and N.~N. Ural'tseva}, {\em Linear
  and quasi-linear equations of parabolic type}, vol.~23, American Mathematical
  Soc., 1968.

\bibitem{lai2022finite}
{\sc C.-C. Lai, F.~Lin, C.~Wang, J.~Wei, and Y.~Zhou}, {\em Finite time blowup
  for the nematic liquid crystal flow in dimension two}, Communications on pure
  and applied mathematics, 75 (2022), pp.~128--196.

\bibitem{lei2014remarks}
{\sc Z.~Lei, D.~Li, and X.~Zhang}, {\em Remarks of global wellposedness of
  liquid crystal flows and heat flows of harmonic maps in two dimensions},
  Proceedings of the American Mathematical Society, 142 (2014), pp.~3801--3810.

\bibitem{leslie1968some}
{\sc F.~M. Leslie}, {\em Some constitutive equations for liquid crystals},
  Archive for Rational Mechanics and Analysis, 28 (1968), pp.~265--283.

\bibitem{li2014global}
{\sc J.~Li}, {\em Global strong and weak solutions to inhomogeneous nematic
  liquid crystal flow in two dimensions}, Nonlinear Analysis: Theory, Methods
  \& Applications, 99 (2014), pp.~80--94.

\bibitem{li2016uniqueness}
{\sc J.~Li, E.~S. Titi, and Z.~Xin}, {\em On the uniqueness of weak solutions
  to the ericksen--leslie liquid crystal model in $\mathbb{R}^2$}, Mathematical
  Models and Methods in Applied Sciences, 26 (2016), pp.~803--822.

\bibitem{li2017global}
{\sc L.~Li, Q.~Liu, and X.~Zhong}, {\em Global strong solution to the
  two-dimensional density-dependent nematic liquid crystal flows with vacuum},
  Nonlinearity, 30 (2017), p.~4062.

\bibitem{lin2010liquid}
{\sc F.~Lin, J.~Lin, and C.~Wang}, {\em Liquid crystal flows in two
  dimensions}, Archive for Rational Mechanics and Analysis, 197 (2010),
  pp.~297--336.

\bibitem{lin2016global}
{\sc F.~Lin and C.~Wang}, {\em Global existence of weak solutions of the
  nematic liquid crystal flow in dimension three}, Communications on Pure and
  Applied Mathematics, 69 (2016), pp.~1532--1571.

\bibitem{liu2016global}
{\sc Q.~Liu, S.~Liu, W.~Tan, and X.~Zhong}, {\em Global well-posedness of the
  2d nonhomogeneous incompressible nematic liquid crystal flows}, Journal of
  Differential Equations, 261 (2016), pp.~6521--6569.

\bibitem{lu2019local}
{\sc B.~L{\"u} and S.~Song}, {\em On local strong solutions to the
  three-dimensional nonhomogeneous navier--stokes equations with
  density-dependent viscosity and vacuum}, Nonlinear Analysis: Real World
  Applications, 46 (2019), pp.~58--81.

\bibitem{oseen1933theory}
{\sc C.~Oseen}, {\em The theory of liquid crystals}, Transactions of the
  Faraday Society, 29 (1933), pp.~883--899.

\bibitem{paicu2013global}
{\sc M.~Paicu, P.~Zhang, and Z.~Zhang}, {\em Global unique solvability of
  inhomogeneous navier-stokes equations with bounded density}, Communications
  in Partial Differential Equations, 38 (2013), pp.~1208--1234.

\bibitem{seregin2014lecture}
{\sc G.~Seregin}, {\em Lecture notes on regularity theory for the Navier-Stokes
  equations}, World Scientific, 2014.

\bibitem{wang2014global}
{\sc M.~Wang and W.~Wang}, {\em Global existence of weak solution for the 2-d
  ericksen--leslie system}, Calculus of Variations and Partial Differential
  Equations, 51 (2014), pp.~915--962.

\bibitem{wang2016uniqueness}
{\sc M.~Wang, W.~Wang, and Z.~Zhang}, {\em On the uniqueness of weak solution
  for the 2-d ericksen-leslie system}, Discrete Contin. Dyn. Syst. Ser. B, 21
  (2016), pp.~919--941.

\bibitem{wang2013well}
{\sc W.~Wang, P.~Zhang, and Z.~Zhang}, {\em Well-posedness of the
  ericksen--leslie system}, Archive for Rational Mechanics and Analysis, 210
  (2013), pp.~837--855.

\bibitem{wen2011solutions}
{\sc H.~Wen and S.~Ding}, {\em Solutions of incompressible hydrodynamic flow of
  liquid crystals}, Nonlinear Analysis: Real World Applications, 12 (2011),
  pp.~1510--1531.

\bibitem{ye2024existence}
{\sc X.~Ye and M.~Zhu}, {\em Existence and decay of global strong solutions to
  the nonhomogeneous incompressible liquid crystal system with vacuum and
  density-dependent viscosity}, Communications in Mathematical Sciences, 22
  (2024), pp.~257--283.

\bibitem{zhang2015global}
{\sc J.~Zhang}, {\em Global well-posedness for the incompressible
  navier--stokes equations with density-dependent viscosity coefficient},
  Journal of Differential Equations, 259 (2015), pp.~1722--1742.

\bibitem{zhang2020global}
{\sc P.~Zhang}, {\em Global fujita-kato solution of 3-d inhomogeneous
  incompressible navier-stokes system}, Advances in Mathematics, 363 (2020),
  p.~107007.

\end{thebibliography}

\end{document}